\documentclass[11pt]{article}
\usepackage{listings}
\usepackage{pdflscape}
\usepackage{amsfonts}
\usepackage{epsfig}
\usepackage{lscape}
\usepackage{amsmath,amssymb,amsthm,mathtools}
\usepackage{graphicx,subcaption}
\usepackage{color}
\usepackage{placeins}
\usepackage{url}
\usepackage{cases,empheq}
\usepackage{longtable}
\usepackage{supertabular}

\allowdisplaybreaks[4]
\usepackage{cite}
\usepackage{algorithm,algpseudocode,setspace,float}
\usepackage{booktabs,multirow,array,threeparttable}
\usepackage{enumitem}
\usepackage{lipsum}
\usepackage{hyperref}
\usepackage{cleveref}
\usepackage{svg}



\newtheorem{theorem}{Theorem}
\newtheorem{assumption}{Assumption}
\newtheorem{lemma}{Lemma}
\newtheorem{remark}{Remark}

\newtheorem{proposition}{Proposition}

\newcommand{\mc}{\mathcal}
\newcommand{\mb}{\mathbb}
\newcommand{\mr}{\mathrm}
\newcommand{\bs}{\boldsymbol}
\newcommand{\col}{\operatorname{col}}
\newcommand{\prox}{\operatorname{prox}}

\newcommand{\blkdiag}{\operatorname{blkdiag}}
\newcommand{\argmin}{\operatornamewithlimits{arg\,min}}

\usepackage{caption}
\begin{document}

\title{\bf dHPR: A Distributed Halpern Peaceman--Rachford Method for Non-smooth Distributed Optimization Problems \footnotemark[1]}
\author{Zhangcheng Feng\footnotemark[2], \quad Defeng Sun\footnotemark[3], \quad Yancheng Yuan\footnotemark[4], \quad Guojun Zhang\footnotemark[5]}
\date{\today}
\maketitle

\renewcommand{\thefootnote}{\fnsymbol{footnote}}
\footnotetext[1]{{\bf Funding:} The work of Defeng Sun was supported by the Research Center for Intelligent Operations Research, RGC {Senior Research Fellow Scheme No. SRFS2223-5S02}, and  {GRF Project No. 15307822}. The work of Yancheng Yuan was supported by the RGC Early Career Scheme (Project No.
25305424) and the Research Center for Intelligent Operations Research.}
\footnotetext[2]{Department of Applied Mathematics, The Hong Kong Polytechnic University, Hung Hom, Hong Kong ({\tt zhangcheng.feng@connect.polyu.hk}).}
\footnotetext[3]{Department of Applied Mathematics, The Hong Kong Polytechnic University, Hung Hom, Hong Kong ({\tt defeng.sun@polyu.edu.hk}).}
\footnotetext[4]{Department of Applied Mathematics, The Hong Kong Polytechnic University, Hung Hom, Hong Kong (\textbf{Corresponding author}. {\tt yancheng.yuan@polyu.edu.hk}).}
\footnotetext[5]{Department of Applied Mathematics, The Hong Kong Polytechnic University, Hung Hom, Hong Kong ({\tt guojun.zhang@connect.polyu.hk}).}
\renewcommand{\thefootnote}{\arabic{footnote}}

\begin{abstract}
     This paper introduces the distributed Halpern Peaceman--Rachford (dHPR) method, an efficient algorithm for solving distributed convex composite optimization problems with non-smooth objectives, which achieves a non-ergodic \( O(1/k) \) iteration complexity regarding Karush--Kuhn--Tucker residual. By leveraging the symmetric Gauss--Seidel decomposition, the dHPR effectively decouples the linear operators in the objective functions and consensus constraints while maintaining parallelizability and avoiding additional large proximal terms, leading to a decentralized implementation with provably fast convergence. The superior performance of dHPR is demonstrated through comprehensive numerical experiments on distributed LASSO, group LASSO, and $L_1$-regularized logistic regression problems.
\end{abstract}

\newcommand{\keywords}[1]{%
  \par\vspace{0.5\baselineskip}%
  \noindent{\textbf{Keywords:} #1}
}


\section{Introduction}\label{sec:introduction}

{T}{his} paper focuses on designing a provably efficient algorithm for solving the following convex composite distributed optimization problems (DOPs) over a multi-agent network of $N$ agents:
\begin{align}\label{eq:dops}
    \min_{x\in\mb R^p}\quad \sum_{i=1}^N ~ f_i(A_ix)+r_i(x),
\end{align}
where, for each $i=1,\ldots, N$, the \(i\)-th agent privately holds a data matrix \(A_i\in\mathbb{R}^{m_i\times p}\) and proper, closed, convex (possibly non-smooth) functions \(f_i:\mathbb{R}^{m_i}\to(-\infty,+\infty]\) and \(r_i:\mathbb{R}^p\to(-\infty,+\infty]\). Formulation \eqref{eq:dops} encompasses a wide class of important distributed optimization problems, such as LASSO \cite{tibshirani1996lasso}, group LASSO \cite{yuan2006glasso}, and regularized logistic regression \cite{lee2006logistic}, that underpin applications in decentralized machine learning \cite{lian2017can}, sensor networks \cite{rabbat2004sensor}, control systems \cite{cao2012mac}, and so on.

Designing efficient decentralized algorithms for solving \eqref{eq:dops} has been a key research goal in distributed optimization, with significant progress over the last decade. Early work primarily treated the fully smooth case---each $f_i$ and $r_i$ are differentiable with Lipschitz‐continuous gradients \cite{nedic2009dgd,shi2014dadmm,ling2015dlm,shi2015extra,yuan2016dgd}. Later, to handle smooth $f_i$ and proximable (possibly non-smooth) regularizers $r_i$, several distributed proximal-splitting methods have been developed~\cite{chang2014ladmm,shi2015pgextra,aybat2017pgadmm,li2019nids,zhang2021pad,condat2022distributed,guo2023dipgm}. Among these, \emph{PG-EXTRA}~\cite{shi2015pgextra}---a special case of the Condat--V\~u scheme~\cite{condat2013cv,vu2013cv}---achieves an ${O}(1/\sqrt{k})$ rate for the averaged Karush--Kuhn--Tucker (KKT) residual in the general convex setting. The network-independent \emph{NIDS} algorithm~\cite{li2019nids}, closely related to \emph{PD3O}~\cite{yan2018pd3o}, admits step sizes independent of the graph’s spectral gap and achieves a non-ergodic $o(1/\sqrt{k})$ KKT residual rate.

For general composition of non-smooth functions $f_i$ with linear operators $A_i$, the proximal mapping $\prox_{\,f_i\circ A_i}$ is typically unavailable; this motivates introducing auxiliary variables $y_i$ to decouple $A_i$ from $f_i$ and local copies $x_i$ to impose consensus. Consequently, we can recast \eqref{eq:dops} as the following equivalent consensus–decoupling reformulation (see, e.g.,~\cite{guo2023disa,zhou2024dpadmm}):
\begin{equation}\label{eq:dop1}
\begin{aligned}
    \min_{\{(x_i,\,y_i)\}_{i=1}^N}\quad &\sum_{i=1}^N~ f_i(y_i)+r_i(x_i), \\ 
    \text{s.t.}\quad &x_1=\cdots=x_N, \\
    &A_ix_i-y_i=\bs0,\quad i=1,\,\ldots,\,N.
\end{aligned}
\end{equation}  
The alternating direction method of multipliers (ADMM) and its variants have become widely used and extensively studied for solving~\eqref{eq:dop1}; see, e.g.,~\cite{guo2023disa,zhou2024dpadmm,latafat2019thripddist,li2021pdfpdist,li2022tpus}.
Note that the non-ergodic rates of ADMM-type methods are typically ${O}(1/\sqrt{k})$ in terms of the KKT residuals \cite{cui2016convergence}, which makes it difficult to obtain high-accuracy solutions efficiently (see also Table~\ref{tab:alg_comparison} for a summary of rates). This motivates us to develop an accelerated ADMM framework tailored to~\eqref{eq:dop1}, and thus to \eqref{eq:dops}---that handles general non-smooth $f_i$ and $r_i$.

\begin{table}[htbp]
\centering
\renewcommand{\arraystretch}{1.3}
\resizebox{\textwidth}{!}{
\begin{threeparttable}
\begin{tabular}{lccc}
\toprule
\textbf{Algorithm} & 
\textbf{Convergence Rate} & \textbf{Metric} \\
\midrule
{PG-EXTRA \cite{shi2015pgextra}}          & Averaged / best $O(1/\sqrt{k})$ & KKT residual         \\
{NIDS \cite{li2019nids}}                  & Non-ergodic $o(1/\sqrt{k})$     & KKT residual         \\
{DISA \cite{guo2023disa}}                 & Averaged / best $O(1/\sqrt{k})$ & KKT residual         \\
{DP-ADMM \cite{zhou2024dpadmm}}           & Ergodic $O(1/k)$                & Feasibility violations \& primal objective error \\ 
\textbf{dHPR (This work)}                 & \textbf{Non-ergodic ${O}(1/k)$} & \textbf{KKT residual \& dual objective error} \\ 
\bottomrule
\end{tabular}
\end{threeparttable}
}
\caption{Complexity results of selected distributed composite optimization algorithms.}
\label{tab:alg_comparison}
\end{table}

Recently, some progress has been made on accelerating preconditioned ADMM (pADMM)~\cite{zhang2022hpr,yang2025accpadmm,sun2025accpadmm,zhang2025hot}---in particular, by integrating Halpern iteration~\cite{halpern1967fixed,Lieder2021halpern} into (semi-proximal) Peaceman--Rachford (PR) schemes \cite{eckstein1992dr,lions1979splitting}. Specifically, Zhang et al. \cite{zhang2022hpr} developed the Halpern Peaceman--Rachford (HPR) method (without proximal terms) by incorporating Halpern iteration into the PR method, achieving an ${O}(1/k)$ iteration complexity in terms of both KKT residual and objective error. Furthermore, Sun et al. \cite{sun2025accpadmm} reformulated the semi-proximal PR method as a degenerate proximal point method \cite{bredies2022dppa} with a positive semidefinite preconditioner, and employed Halpern iteration to derive the HPR method with semi-proximal terms, which also attains an ${O}(1/k)$ iteration complexity.

Although HPR methods have shown promise on large-scale centralized models such as linear programming~\cite{chen2024hprlp} and convex quadratic programming~\cite{chen2025hprqp}, a distributed counterpart with theoretical guarantees remains unexplored. Another challenge is to design proximal terms that simultaneously handle the consensus constraint and the local coupling $A_i x_i = y_i$ in \eqref{eq:dop1} so that each per-agent subproblem admits a solution with low computational cost. A common remedy is to introduce a large proximal term~\cite{esser2010general,chambolle2011first} to linearize the coupled quadratic term in the subproblems, but this often slows down the convergence in practice. To address these challenges, we develop a distributed HPR (\textbf{dHPR}) method for \eqref{eq:dop1}---and hence for \eqref{eq:dops}---with the following features:
\begin{enumerate}
    \item \textbf{A distributed HPR with fast convergence rates.} We present dHPR, to our knowledge, the first distributed realization of the HPR algorithm for the consensus-decoupling model \eqref{eq:dop1}, and establish a non-ergodic ${O}(1/k)$ convergence for both the KKT residual and the objective error---improving on the typical non-ergodic ${O}(1/\sqrt{k})$ guarantees of related methods.
    \item \textbf{sGS-based decoupling and cheap per-agent updates.} Incorporating a symmetric Gauss--Seidel (sGS) decomposition ~\cite{li2016schur,li2019block}, dHPR decouples the linear operators $A_i$ from the consensus constraint, yielding closed-form, parallelizable per-agent updates and avoiding introducing a large proximal term.
    \item \textbf{Empirical superiority.} Extensive experiments show that dHPR achieves superior convergence compared with several state-of-the-art distributed optimization methods.
\end{enumerate}

The remainder of the paper is organized as follows. Section~\ref{sec:prob} formalizes the model and introduces some assumptions. Section~\ref{sec:dhpr} introduces the details of the proposed dHPR algorithm, including its convergence rates and an efficient distributed implementation. Detailed numerical results are shown in Section~\ref{sec:numerical} to demonstrate the superior performance of the dHPR algorithm for solving a wide class of convex composite distributed optimization problems. We conclude the paper in Section~\ref{sec:conclusion}.

\vspace{1em}
\noindent\textbf{Notations.}
For any given positive integer $p$, $\bs0_p$ (or $\bs1_p$) is the all-zero (or all-one) vector in $\mb R^p$, and $I_p \in \mb R^{p\times p}$ is the identity matrix. We omit the sub-script (i.e., $\bs0,\,\bs1,\,I$) to denote a vector or matrix with appropriate dimension. For any proper closed convex function $f:\mb R^p\rightarrow(-\infty,\,+\infty]$, we denote 
its subdifferential as the set-valued operator $\partial f:\mb R^p\rightrightarrows\mb R^p:\, x \mapsto \{\xi\in\mb R^p:\,f(y)-f(x)\geq\langle \xi,\, y-x \rangle,\,\forall y\in\mb R^p\}$, its Fenchel conjugate as $f^*(x):=\sup_{y}\{\langle y,\,x \rangle - f(y)\}$, $x\in\mb R^p$, and its proximal mapping as $\prox_{\tau f}(x):=\argmin_{y\in\mb R^p}\{f(y)+\frac{1}{2\tau}\|x-y\|^2\}$, $x\in\mb R^p$, $\tau>0$. Given $N$ vectors $x_1,\,\ldots,\,x_N$, denote $\col(x_1,\,\ldots,\,x_N):=(x_1^\top,\,\ldots,\,x_N^\top)^\top$. Given $N$ matrices $A_1,\,\ldots,\,A_N$, $\mr{blkdiag}(A_1,\,\ldots,\,A_N)$ is the block diagonal matrix with diagonal elements $A_1,\,\ldots,\,A_N$. $\otimes$ denotes the Kronecker product. Given a matrix $V\in\mb R^{m\times n}$, $\ker(V):=\{x\in\mb R^n:\,Vx=\bs0_m\}$. Given a vector $x\in\mb R^n$, $\mr{span}(x):=\{kx:\,k\in\mb R\}$. Given a symmetric matrix $A$, $A\succ\bs0$ means that $A$ is positive definite. For any symmetric and positive semidefinite matrix $\mc M\in\mb R^{p\times p}$, denote $\|x\|_{\mc M}:=\sqrt{\langle x,\,x \rangle_{\mc M}}=\sqrt{\langle x,\,\mc Mx \rangle}$ for any $x\in\mb R^p$.

\section{Problem Formulation}\label{sec:prob}

In this section, we introduce the communication graph framework and reformulate DOPs to facilitate the implementation of our proposed dHPR algorithms, along with some standard assumptions.

\subsection{The Communication Graph}

In distributed networks, agents communicate with each other through a graph denoted by $\mc G=(\mc N,\,\mc E)$, where $\mc N=\{1,\,\ldots,\,N\}$ is the set of nodes (agents) and $\mc E$ is the set of edges. Suppose $i,\,j\in\mc N$, if $i$ can receive $j$'s information, then $(i,\,j)\in\mc E$. 
A \textit{directed path} from $i_1$ to $i_k$ is a sequence of edges $\{(i_1,\,i_2),\,\ldots,\,(i_{k-1},\,i_k)\}$ with distinct nodes $\{i_r\}_{r=1}^k$. The graph is \textit{strongly connected} if there exists at least one directed path between any two nodes in the graph. Additionally, a weighted adjacency matrix ${W=[a_{ij}]}\in\mb R^{N\times N}$ of $\mc G$ satisfies: ${a_{ij}>0}$ if $(i,\,j)\in\mc E$ or $i=j$, and ${a_{ij}=0}$ otherwise. We adopt the following standard assumption regarding the communication graph, which is widely used in distributed optimization literature \cite{shi2015pgextra,li2019nids,guo2023dipgm,zhang2021pad,guo2023disa,zhou2024dpadmm}.
\begin{assumption}\label{ass:graph}
    The communication graph $\mc G$ is strongly connected and the weighted adjacency matrix $W$ is symmetric and doubly stochastic, i.e., $W\bs 1_N=\bs 1_N$, $\bs1_N^\top W=\bs 1_N^\top$.
\end{assumption}
Note that the matrix $I-W\in\mb R^{N\times N}$ is symmetric and positive semidefinite under Assumption \ref{ass:graph}.

\subsection{Reformulations of DOPs and dual DOPs}

Denote
\begin{equation}
\begin{aligned}
    m & :=\sum_{i=1}^Nm_i, \quad\bs x:=\col(x_1,\,\ldots,\,x_N)\in\mb R^{Np},\quad \bs y:=\col(y_1,\,\ldots,\,y_N)\in\mb R^m, \\ 
    f(\bs y) &:=\sum_{i=1}^Nf_i(y_i),\quad r(\bs x) :=\sum_{i=1}^Nr_i(x_i),\quad \bs A:= \blkdiag(A_1,\,\ldots,\,A_N)\in\mb R^{m\times Np}.
\end{aligned}
\end{equation}
 Define $\bs U:=\sqrt{I-W}\otimes I_p\in\mb R^{Np\times Np}$, which satisfies $\ker(\sqrt{I-W})=\mr{span}(\bs 1_N)$ under Assumption \ref{ass:graph}. Then, we can equivalently reformulate problem \eqref{eq:dop1} as:
\begin{equation}\label{eq:dop2}
    \begin{aligned}
        \min_{(\bs x,\,\bs y)\in\mb R^{Np}\times\mb R^m}\quad & f(\bs y)+r(\bs x), \\ 
        \text{s.t.}\quad & \bs U\bs x=\bs0, \\
        & \bs A\bs x-\bs y=\bs0.
    \end{aligned}
\end{equation}
The dual problem of \eqref{eq:dop2} is
\begin{equation}\label{eq:dualdop}
    \begin{aligned}
        \min_{(\bs z,\,\bs w,\,\bs v)\in\mb R^m\times\mb R^{Np}\times\mb R^{Np}}\quad & f^*(\bs z)+r^*(\bs v) \\
        \text{s.t.}\quad & \bs A^\top\bs z+\bs U^\top\bs w+\bs v=\bs0.
    \end{aligned}
\end{equation}
For any given parameter $\sigma>0$, the augmented Lagrangian function of \eqref{eq:dualdop} is defined, for any $(\bs z,\,\bs w,\,\bs v,\,\bs x)\in\mb R^m\times\mb R^{Np}\times\mb R^{Np}\times\mb R^{Np}$, as follows:
\begin{align}\label{eq:lagrangian}
    L_\sigma(\bs z,\bs w,\bs v;\bs x):=&f^*(\bs z)+r^*(\bs v)-\left\langle \bs x,\,\bs A^\top\bs z+\bs U^\top\bs w+\bs v \right\rangle + \frac{\sigma}{2}\left\| \bs A^\top\bs z+\bs U^\top\bs w+\bs v \right\|^2.
\end{align}
For notational convenience, we denote
\begin{equation}
    \begin{aligned}
        \bs u &:=(\bs z,\,\bs w,\,\bs v,\,\bs x),\quad
        \mc U :=\mb R^m\times\mb R^{Np}\times\mb R^{Np}\times\mb R^{Np}.
    \end{aligned}
\end{equation}
According to \cite[Cor. 28.3.1]{rockafellar1997ca}, a point $(\bs z^*,\,\bs w^*,\,\bs v^*)\in\mb R^m\times\mb R^{Np}\times\mb R^{Np}$ is an optimal solution to problem \eqref{eq:dualdop} if there exists $\bs x^*\in\mb R^{Np}$ such that the KKT system below is satisfied:
\begin{equation}\label{eq:kkt}
    \begin{cases}
        \bs0\in\partial f^*(\bs z^*)-\bs A\bs x^*, \\
        \bs0=-\bs U\bs x^*, \\
        \bs0\in\partial r^*(\bs v^*)-\bs x^*, \\
        \bs0=-\bs A^\top\bs z^*-\bs U^\top\bs w^*-\bs v^*.
    \end{cases}
\end{equation}
Now, we make the following assumption.
\begin{assumption}\label{ass:kkt}
    The KKT system \eqref{eq:kkt} has a nonempty solution set.
\end{assumption}
Under Assumption \ref{ass:kkt}, solving problems \eqref{eq:dop2} and \eqref{eq:dualdop} is equivalent to finding a $\bs u^*\in\mc U$ such that $\bs0\in\mc K\bs u^*$, where the maximally monotone operator $\mc K$ is defined by: $\forall \bs u=(\bs z,\,\bs w,\,\bs v,\,\bs x)\in\mc U$,
\begin{equation}\label{eq:kktmap}
    \mc K\bs u :=
    \begin{pmatrix}
        \partial f^*(\bs z)-\bs A\bs x \\
        -\bs U\bs x \\
        \partial r^*(\bs v)-\bs x \\
        \bs A^\top\bs z+\bs U^\top\bs w+\bs v
    \end{pmatrix}.
\end{equation}

\section{dHPR: A distributed HPR method}\label{sec:dhpr}

This section presents the dHPR method for solving DOPs. We first detail the core algorithmic framework, then provide comprehensive convergence analysis, and conclude with a practical and efficient implementation.

\subsection{dHPR}

Denote $\bar{\bs u}:=(\bar{\bs z},\,\bar{\bs w},\,\bar{\bs v},\,\bar{\bs x})$ and $\hat{\bs u}:=(\hat{\bs z},\,\hat{\bs w},\,\hat{\bs v},\,\hat{\bs x})$. The HPR method with semi-proximal terms, which corresponds to the accelerated pADMM with $\rho=2$ and $\alpha=2$ proposed in \cite{sun2025accpadmm}, is presented in Algorithm~\ref{alg:sphpr} for solving the dual problem \eqref{eq:dualdop}.

\begin{algorithm}
\caption{A semi-proximal HPR method for solving the dual problem \eqref{eq:dualdop}}
\label{alg:sphpr}
\begin{algorithmic}[1] 
\Require Choose a symmetric positive semidefinite matrix $\mc T\in\mb R^{(m+Np)\times(m+Np)}$, $\bs u^0=(\bs z^0,\,\bs w^0,\,\bs v^0,\,\bs x^0)\in\mc U$ and set $\sigma>0$.
\For{$k=0,\,1,\,\cdots$}
    \State $\bar{\bs v}^{k+1}=\argmin\limits_{\bs v\in\mb R^{Np}} \left\{L_\sigma(\bs z^k,\bs w^k,\bs v;\bs x^k)\right\}$
    \State $\bar{\bs x}^{k+1}={\bs x}^k-\sigma(\bs A^\top\bs z^k+\bs U^\top\bs w^k+\bar{\bs v}^{k+1})$
    \State
    $(\bar{\bs z}^{k+1},\bar{\bs w}^{k+1})=\argmin\limits_{(\bs z,\bs w)\in\mb R^{m}\times\mb R^{Np}} \left\{L_\sigma(\bs z,\bs w,\bar{\bs v}^{k+1};\bar{\bs x}^{k+1})+\dfrac{1}{2}\left\| (\bs z,\bs w)-(\bs z^k,\bs w^k) \right\|^2_{\mc T}\right\}$
    \State $\hat{\bs u}^{k+1}=2\bar{\bs u}^{k+1}-\bs u^k$
    \State $\bs u^{k+1}=\dfrac{1}{k+2}\bs u^0+\dfrac{k+1}{k+2}\hat{\bs u}^{k+1}$
\EndFor
\State \Return $\bar{\bs u}^{k+1}$
\end{algorithmic}
\end{algorithm}
Note that the main computational bottleneck in Algorithm~\ref{alg:sphpr} lies in solving the subproblem involving the variables $(\bs z,\,\bs w)$ (Line 4). 
A key step is choosing a suitable proximal operator $\mathcal{T}$ to simplify this subproblem. 
To address this difficulty and enable decentralized updates, we employ the sGS technique \cite{li2016schur,li2019block} to decouple $\bs z$ and $\bs w$. 
Specifically, define the symmetric positive semidefinite matrix $\mc S$ and the sGS operator $\hat{\mc S}$ as
\begin{align}\label{eq:S&Shat}
    \mc S=\sigma\begin{bmatrix} \mc S_z & \bs0 \\ \bs0 & \mc S_w \end{bmatrix}, \,
    \hat{\mc S}=\begin{bmatrix} {\sigma}\bs A\bs U(\mc S_w+\bs U^2)^{-1}\bs U\bs A^\top & \bs0 \\ \bs0 & \bs0 \end{bmatrix},
\end{align}
where $\mc S_z\in\mb R^{m\times m}$ and $\mc S_w\in\mb R^{Np\times Np}$ are symmetric positive semidefinite matrices such that $\mc S_z+\bs A\bs A^\top\succ\bs0$ and $\mc S_w+\bs U^2\succ\bs0$, respectively. The following result shows that incorporating $\hat{\mc S}$ enables efficient decoupled updates of $(\bar{\bs z}^{k+1},\,\bar{\bs w}^{k+1})$ in Algorithm~\ref{alg:sphpr} for all $k\geq0$.
\begin{proposition}[{\cite[Thm. 1]{li2019block}}]\label{pro:sgs}
    Let $\mc T=\mc S+\hat{\mc S}$ with $\mc S$ and $\hat{\mc S}$ given in \eqref{eq:S&Shat}. Then for any $k\geq0$, the update of $(\bar{\bs z}^{k+1},\, \bar{\bs w}^{k+1})$ in Algorithm~\ref{alg:sphpr}, i.e.,
    \begin{align*}
        (\bar{\bs z}^{k+1},\bar{\bs w}^{k+1})=\argmin\limits_{(\bs z,\bs w)\in\mb R^{m}\times\mb R^{Np}} \left\{L_\sigma(\bs z,\bs w,\bar{\bs v}^{k+1};\bar{\bs x}^{k+1})+\dfrac{1}{2}\left\| (\bs z,\bs w)-(\bs z^k,\bs w^k) \right\|^2_{\mc T}\right\}
    \end{align*}
    is equivalent to the following updates:
    \begin{align*}
    \begin{cases}
        \notag\bar{\bs w}^{k+\frac{1}{2}}=\argmin\limits_{\bs w\in\mb R^{Np}} \left\{L_\sigma(\bs z^k,\bs w,\bar{\bs v}^{k+1};\bar{\bs x}^{k+1})+\dfrac{\sigma}{2}\left\| \bs w-\bs w^k \right\|^2_{\mc S_w}\right\}, \\
        \notag\bar{\bs z}^{k+1}=\argmin\limits_{\bs z\in\mb R^m} \left\{L_\sigma(\bs z,\bar{\bs w}^{k+\frac12},\bar{\bs v}^{k+1};\bar{\bs x}^{k+1})+\dfrac{\sigma}{2}\left\| \bs z-\bs z^k \right\|^2_{\mc S_z}\right\}, \\
        \notag\bar{\bs w}^{k+1}=\argmin\limits_{\bs w\in\mb R^{Np}} \left\{L_\sigma(\bar{\bs z}^{k+1},\bs w,\bar{\bs v}^{k+1};\bar{\bs x}^{k+1})+\dfrac{\sigma}{2}\left\| \bs w-\bs w^k \right\|^2_{\mc S_w}\right\}.
    \end{cases}
    \end{align*}
    Moreover, $\mc T+\sigma\bs A_U^\top\bs A_U\succ\bs0$, where $\bs A_U:=\begin{bmatrix} \bs A^\top & \bs U^\top \end{bmatrix}\in\mb R^{Np\times(m+Np)}$.
\end{proposition}

To get the decentralized updates, we choose
\begin{align}\label{eq:Sz&Sw}
    \mc S_z=\blkdiag(\mc S_z^1,\,\dots,\,\mc S_z^N) \text{ and }\mc S_w=\lambda_UI-\bs U^2
\end{align}
with $\mc S_z^i\in\mb R^{m_i\times m_i}$ being symmetric positive semidefinite for all $i\in\mc N$ and $\lambda_U\geq1-\lambda_{\min}(W)$. Furthermore, denote
\begin{equation}
    \begin{aligned}
        \bs z &=\col(z_1,\,\ldots,\,z_N)\in\mb R^m,\,z_i\in\mb R^{m_i},\,i\in\mc N, \\
        \bs v &=\col(v_1,\,\ldots,\,v_N)\in\mb R^{Np},\,v_i\in\mb R^p,\,i\in\mc N, \\
        \bs s &:= \bs U^\top\bs w=\col(s_1,\,\dots,\,s_N)\in\mb R^{Np},\,s_i\in\mb R^p,\,i\in\mc N, \\
        \bs u_s &:= (\bs z,\,\bs s,\,\bs v,\,\bs x),\,u_{s,i}:=(z_i,\,s_i,\,v_i,\,x_i),\,i\in\mc N,
    \end{aligned}
\end{equation}
and $\bar u_{s,i}$ and $\hat u_{s,i}$ are defined similarly for all $i\in\mc N$. Given \eqref{eq:S&Shat} and \eqref{eq:Sz&Sw}, by letting $\mc T=\mc S+\hat{\mc S}$ in Algorithm~\ref{alg:sphpr}, we obtain the dHPR method in Algorithm~\ref{alg:hprdop}, which proceeds iteratively with all agents updating in a decentralized and parallel manner.

\begin{algorithm}
\caption{dHPR: A distributed HPR method for solving the dual DOP \eqref{eq:dualdop}}
\label{alg:hprdop}
\begin{algorithmic}[1] 
\Require
Given $W=[a_{ij}]\in\mathbb{R}^{N\times N}$ and $\sigma>0$, each agent $i\in\mathcal{N}$ selects $\lambda_U \geq 1-\lambda_{\min}(W)$, a symmetric positive semidefinite matrix $\mathcal{S}_z^i \in \mathbb{R}^{m_i\times m_i}$, and initializes $(\bs z^0,\,\bs w^0,\,\bs v^0,\,\bs x^0)\in\mc U$ and $\bs s^0=\bs U^\top\bs w^0$.
\For{$k=0,\,1,\,\cdots$} (all agents $i\in\mc N$ perform in parallel)
    \State
    $\bar{v}_i^{k+1}=\argmin\limits_{v\in\mb R^p}\left\{r^*_i(v)-\langle x_i^k,v \rangle +\dfrac{\sigma}2\left\| A_i^\top z_i^k+s_i^k+v \right\|^2\right\}$
    \vspace{1pt}
    \State $\bar{x}_i^{k+1}=x_i^k-\sigma\left( A_i^\top z_i^k+s_i^k+\bar{v}_i^{k+1} \right)$
    \vspace{1pt}
    \State $\bar{s}_i^{k+\frac{1}{2}}=s_i^k+\dfrac{1}{\sigma\lambda_U}\Big[ 2\bar{x}_i^{k+1}-x_i^k-\sum\limits_{j=1}^Na_{ij}(2\bar{x}_j^{k+1}-x_j^k) \Big]$
    \vspace{1pt}
    \State
    $\bar{z}_i^{k+1}=\argmin\limits_{z\in\mb R^{m_i}}\left\{f^*_i(z)-\langle \bar{x}_i^{k+1},A_i^\top z\rangle+\dfrac{\sigma}2\left\| A_i^\top z+\bar{s}_i^{k+\frac12}+\bar{v}_i^{k+1} \right\|^2+\dfrac{\sigma}{2}\left\| z-z_i^k \right\|^2_{\mc S_z^i}\right\}$
    \vspace{1pt}
    \State $\bar{s}_i^{k+1}=\bar{s}_i^{k+\frac12}+\dfrac{1}{\lambda_U}\left[ A_i^\top(z_i^k-\bar{z}_i^{k+1})-\sum\limits_{j=1}^Na_{ij}A_j^\top(z_j^k-\bar{z}_j^{k+1}) \right]$
    \vspace{1pt}
    \State $\hat{u}^{k+1}_{s,i} = 2\bar{u}^{k+1}_{s,i}-u^{k}_{s,i}$
    \vspace{2pt}
    \State $u^{k+1}_{s,i} = \dfrac{1}{k+2}u^{0}_{s,i} + \dfrac{k+1}{k+2}\hat{u}^{k+1}_{s,i}$
\EndFor
\State \Return $\{\bar{u}^{k+1}_{s,i}\}_{i=1}^N$
\end{algorithmic}
\end{algorithm}

\begin{remark}[sGS decomposition]\label{rmk:sgs}
    Due to the composite structure of $\bs A$ and $\bs U$, a natural approach applying Algorithm~\ref{alg:sphpr} to solve the problem \eqref{eq:dualdop} typically requires a large proximal operator of the form $\mc T=\sigma(\lambda_{AU}I-\bs A_U^\top\bs A_U)$ with $\lambda_{AU}\geq\lambda_{\max}(\bs A_U)$ to simplify the subproblem with respect to $(\bs z,\, \bs w)$~\cite{esser2010general,chambolle2011first}. Note that the resulting large spectral norm $\|\mc T\|$ can significantly slow convergence. In contrast, Algorithm~\ref{alg:hprdop} incorporates the sGS decomposition to decouple $\bs A$ and $\bs U$, yielding a computationally efficient proximal operator $\mc T=\mc S+\hat{\mc S}$ that leads to a decentralized implementation with fast convergence (see Appendix~\ref{app:dual-lhpr} for more details).
\end{remark}

\begin{remark}[Connection with existing algorithms]\label{rmk:connection}
    {When $f_i\equiv0$ for all $i\in\mc N$, the update of $\bar{v}_i^{k+1}$ in Algorithm~\ref{alg:hprdop} is given by
    \begin{equation*}
    \begin{aligned}
        \bar{v}_i^{k+1}
        &=\frac{x_i^k}{\sigma} -s_i^k-\frac{\prox_{\sigma r_i}(x_i^k-\sigma s_i^k)}{\sigma}.
    \end{aligned}
    \end{equation*}
   Substituting the above equality into the $\bar{x}_i^{k+1}$-update, we obtain the following updates for $\bar{x}_i^{k+1}$ and $\bar{s}_i^{k+1}$ in Algorithm~\ref{alg:hprdop} (with $\lambda_U=2$):
    \begin{align*}
    \begin{dcases}
        \bar{x}_i^{k+1}=\prox_{\sigma r_i}(x_i^k-\sigma s_i^k), \\
        \bar{s}_i^{k+1}=s_i^k+\dfrac{1}{2\sigma}\left[ 2\bar{x}_i^{k+1}-x_i^k-\sum\limits_{j=1}^Na_{ij}(2\bar{x}_j^{k+1}-x_j^k) \right].
    \end{dcases}
    \end{align*}
    These two steps correspond to P-EXTRA \cite[Alg. 2]{shi2015pgextra} and NIDS \cite{li2019nids} when $f_i\equiv0$ for all $i\in\mc N$. Consequently, in this setting, Algorithm~\ref{alg:hprdop} can be seen as an accelerated variant of P-EXTRA and NIDS with relaxation step (Line 7) using Halpern iteration (Line 8).}
\end{remark}

\subsection{Convergence Analysis}

This subsection presents the theoretical results of dHPR. We first state the convergence of dHPR in Theorem~\ref{thm:convergence}; its proof appears in Appendix~\ref{app:thm_convergence}.

\begin{theorem}\label{thm:convergence}
    Suppose that Assumptions \ref{ass:graph}-\ref{ass:kkt} hold. If $\mc S_z^i+A_iA_i^\top\succ\bs0$ for all $i\in\mc N$, then the sequence $\{\bar{\bs u}_s^{k}\}=\{(\bar{\bs z}^k,\,\bar{\bs s}^k,\,\bar{\bs v}^k,\,\bar{\bs x}^k)\}$ generated by Algorithm~\ref{alg:hprdop} converges to the point $(\bs z^*,\,\bs U^\top\bs w^*,\,\bs v^*,\,\bs x^*)$, where $(\bs z^*,\,\bs w^*,\,\bs v^*)$ is a solution to the problem \eqref{eq:dualdop} and $\bs x^*$ is a solution to the problem \eqref{eq:dop2}.
\end{theorem}

To further analyze the complexity of the Algorithm~\ref{alg:hprdop} in terms of the KKT residual and the objective error, we consider the following residual mapping associated with the KKT system \eqref{eq:kkt}, as introduced in \cite{han2018linearadmm}: $\forall \ \bs u_s=(\bs z,\,\bs s,\,\bs v,\,\bs x)\in\mc U$,
\begin{equation}\label{eq:kktresidual}
    \mc R(\bs u_s):=
    \begin{pmatrix}
        \prox_f(\bs z+\bs A\bs x)-\bs A\bs x \\
        \bs U\bs x \\
        \prox_r(\bs v+\bs x)-\bs x \\
        \bs A^\top\bs z+\bs s+\bs v
    \end{pmatrix}.
\end{equation}
Let $\{(\bar{\bs z}^k,\,\bar{\bs v}^k)\}$ be the sequence generated by Algorithm~\ref{alg:hprdop}, to estimate the objective error, $\forall k\geq1$, we define
\begin{equation}\label{eq:dualerror}
    h(\bar{\bs z}^k,\bar{\bs v}^k):= f^*(\bar{\bs z}^k)+r^*(\bar{\bs v}^k)-\big(f^*({\bs z}^*)+r^*({\bs v}^*)\big),
\end{equation}
where $({\bs z}^*,\,\bs w^*,\,{\bs v}^*)$ is a solution to the problem \eqref{eq:dualdop}. In addition, we define a self-adjoint positive semidefinite linear operator $\mc M:\,\mc U\rightarrow\mc U$ as
\begin{equation}\label{eq:M}
    \mc M:=
    \begin{bmatrix}
        \sigma\bs A_U^\top\bs A_U+\mc S+\hat{\mc S} & \bs0 & -\bs A_U^\top \\
        \bs0 & \bs0 &\bs0 \\
        -\bs A_U & \bs0 & \dfrac{1}{\sigma}I_{Np}
    \end{bmatrix}
\end{equation}
with $\mc S$ and $\hat{\mc S}$ given in \eqref{eq:S&Shat} and $\bs A_U$ defined in Proposition \ref{pro:sgs}. We now present the main complexity result regarding the KKT residual and the dual objective error of the dHPR in Theorem \ref{thm:complexity}; the proof is given in Appendix~\ref{app:thm_convergence}.

\begin{theorem}\label{thm:complexity}
    Suppose that Assumptions \ref{ass:graph}-\ref{ass:kkt} hold. Let $\{\bar{\bs u}_s^{k}\}=\{(\bar{\bs z}^k,\,\bar{\bs s}^k,\,\bar{\bs v}^k,\,\bar{\bs x}^k)\}$ and $\{{\bs u}_s^{k}\}=\{({\bs z}^k,\,{\bs s}^k,\,{\bs v}^k,\,{\bs x}^k)\}$ be the sequences generated by Algorithm~\ref{alg:hprdop}, and let ${\bs u}^*=({\bs z}^*,\,{\bs w}^*,\,{\bs v}^*,\,{\bs x}^*)$ be a solution to the KKT system \eqref{eq:kkt}. Define $R_0:=\|\bs u_s^{0}-\bs u^*\|_{\mc M}$ with $\mc M$ given in \eqref{eq:M}. If $\mc S_z^i+A_iA_i^\top\succ\bs0$ for all $i\in\mc N$, then the following complexity bounds hold for all $k\geq0$:
    \begin{gather*}
        \left\| \mc R(\bar{\bs u}_s^{k+1}) \right\| \leq \left( \frac{\sigma\|\bs A_{U}\|+1}{\sqrt{\sigma}}+\left\|\sqrt{\mc S+\hat{\mc S}}\right\|\right)\frac{R_0}{k+1},\\
        \left( -\frac{\|\bs x^*\|}{\sqrt{\sigma}} \right)\frac{R_0}{k+1} \leq h(\bar{\bs z}^{k+1},\bar{\bs v}^{k+1}) \leq \left( 3R_0+\frac{\|\bs x^*\|}{\sqrt{\sigma}} \right)\frac{R_0}{k+1},
    \end{gather*}
    where $\mc R(\cdot)$ and $h(\cdot,\cdot)$ are defined in \eqref{eq:kktresidual} and \eqref{eq:dualerror}, respectively.
\end{theorem}

\subsection{An Efficient Implementation of dHPR}

In this subsection, we present an efficient implementation of Algorithm~\ref{alg:hprdop}. To begin with, we derive the update formulas for each subproblem in Algorithm~\ref{alg:hprdop}. To be specific, for any $k\geq0$, $i\in \mathcal{N}$, the update of $\bar{v}_i^{k+1}$ is given by
\begin{align}\label{eq:updata_v}
    \notag\bar{v}_i^{k+1}&=\argmin_{v\in\mb R^{p}} \left\{ r^*(v)-\langle x_i^k,\, v \rangle + \frac{\sigma}{2}\left\| A_i^\top z_i^k+s_i^k+v \right\|^2 \right\} \\
    \notag&=\prox_{\tfrac{r_i^*}{\sigma}}\left( -A_i^\top z_i^k-s_i^k+\frac{x_i^k}{\sigma} \right) \\
    &=\frac{1}{\sigma}\left[ \phi_i^k-\prox_{\sigma r_i}(\phi_i^k) \right],
\end{align}
where $\phi_i^k=x_i^k-\sigma(A_i^\top z_i^k+s_i^k)$. Then the update of $\bar{x}_i^{k+1}$ becomes
\begin{equation}\label{eq:update_x}
    \bar{x}_i^{k+1}={x}_i^k-\sigma(A_i^\top z_i^k+s_i^k+\bar{v}_i^{k+1})=\prox_{\sigma r_i}(\phi_i^k).
\end{equation}
To simplify the subproblem with respect to $z_i$, we choose
\begin{equation}
    \mc S_z^i=\lambda_A^iI- A_iA_i^\top,
\end{equation}
where $\lambda_A^i\geq\lambda_{\max}(A_iA_i^\top),\,\forall i\in\mc N$.
Then, the update of $\bar{z}_i^{k+1}$ is given by
\begin{align}\label{eq:update_z}
    \notag\bar{z}_i^{k+1} &= \prox_{\tfrac{f^*_i}{\sigma\lambda_A^i}}\left[ -\frac1{\sigma\lambda_A^i}\left( -A_i\bar{x}_i^{k+1}+\sigma A_i(\bar{s}_i^{k+\frac12}+\bar{v}_i^{k+1})-\sigma(\lambda_A^iI-A_iA_i^\top)z_i^k \right)\right] \\
    &= \frac1{\sigma\lambda_A^i}\left[ {\xi}_i^k-\prox_{\sigma\lambda_A^if_i}(\xi_i^k) \right],
\end{align}
where
\begin{align}\label{eq:update_xi}
    \notag\xi_i^k &= \left[ A_i\bar{x}_i^{k+1}-\sigma A_i(\bar{s}_i^{k+\frac12}+\bar{v}_i^{k+1})+\sigma(\lambda_A^iI-A_iA_i^\top)z_i^k \right] \\
    &= A_i\left[ (2\bar{x}_i^{k+1}-x_i^k)-\sigma(\bar{s}_i^{k+\frac12}-s_i^k) \right] + \sigma\lambda_A^i z_i^k.
\end{align}

By combining \eqref{eq:updata_v}-\eqref{eq:update_xi}, we obtain an efficient implementation of dHPR presented in Algorithm~\ref{alg:hprdop2}.

\begin{algorithm}
\caption{An efficient implementation of dHPR}
\label{alg:hprdop2}
\begin{algorithmic}[1] 
\Require Given matrix $W=[a_{ij}]\in\mb R^{N\times N}$ and $\sigma>0$, each agent $i\in\mc N$ selects $\lambda_U\geq1-\lambda_{\min}(W),\,\lambda_A^i\geq\lambda_{\max}(A_iA_i^\top)$, and initializes $(\bs z^0,\,\bs w^0,\,\bs v^0,\,\bs x^0)\in\mc U$ and $\bs s^0=\bs U^\top\bs w^0$.
\For{$k=0,\,1,\,\cdots$} (all agents $i\in\mc N$ perform in parallel)
    \State $\phi_i^k=x_i^k-\sigma(A_i^\top z_i^k+s_i^k)$
    \vspace{1pt}
    \State $\bar{v}_i^{k+1}=\dfrac{1}{\sigma}\left[ \phi_i^k-\prox_{\sigma r_i}(\phi_i^k) \right]$
    \vspace{1pt}
    \State $\bar{x}_i^{k+1}=\prox_{\sigma r_i}\left(\phi_i^k\right)$
    \vspace{1pt}
    \State $\bar{s}_i^{k+\frac{1}{2}}=s_i^k+\dfrac{1}{\sigma\lambda_U}\Big[ 2\bar{x}_i^{k+1}-x_i^k-\sum\limits_{j=1}^Na_{ij}(2\bar{x}_j^{k+1}-x_j^k) \Big]$
    \vspace{1pt}
    \State $\xi_i^k= A_i\Big[ (2\bar{x}_i^{k+1}-x_i^k)-\sigma(\bar{s}_i^{k+\frac12}-s_i^k) \Big] + \sigma\lambda_A^iz_i^k$
    \vspace{1pt}
    \State $\bar{z}_i^{k+1}=\dfrac{1}{\sigma\lambda_A^i}\left[ \xi_i^k-\prox_{\sigma\lambda_A^if_i}(\xi_i^k) \right]$
    \State $\bar{s}_i^{k+1}=\bar{s}_i^{k+\frac12}+\dfrac{1}{\lambda_U}\left[ A_i^\top(z_i^k-\bar{z}_i^{k+1})-\sum\limits_{j=1}^Na_{ij}A_j^\top(z_j^k-\bar{z}_j^{k+1}) \right]$
    \vspace{1pt}
    \State $\hat{u}^{k+1}_{s,i} = 2\bar{u}^{k+1}_{s,i}-u^{k}_{s,i}$
    \vspace{1pt}
    \State $u^{k+1}_{s,i} = \dfrac{1}{k+2}u^{0}_{s,i} + \dfrac{k+1}{k+2}\hat{u}^{k+1}_{s,i}$
\EndFor
\State \Return $\{\bar{u}^{k+1}_{s,i}\}_{i=1}^N$
\end{algorithmic}
\end{algorithm}

\begin{remark}
    We observe that explicit computation of $\bar{v}_i^{k+1}$ can be avoided during the iterative process. The variable $\bar{v}_i^{k+1}$ only needs to be evaluated via \eqref{eq:updata_v} when verifying the stopping criteria, which further enhances the computational efficiency while maintaining the correctness of the algorithm.
\end{remark}

\section{Numerical Experiment}\label{sec:numerical}

This section presents numerical experiments on three distributed optimization problems: LASSO, group LASSO, and $L_1$-regularized logistic regression, to demonstrate the efficiency of dHPR. All numerical experiments were conducted using MATLAB (Version R2024b) on a Windows laptop equipped with an Intel Core i5-1135G7 processor (2.40GHz, 4 cores) and 16GB of RAM.

\subsection{Experimental Setup}

\subsubsection{Communication graph}
The random communication graph $\mathcal{G}$ is generated as an undirected connected graph consisting of $N$ nodes with $\tfrac{\iota N(N-1)}{2}$ edges, where $\iota\in(0,1]$ is the connectivity ratio. The weighted adjacency matrix $W=[a_{ij}]\in\mathbb{R}^{N\times N}$ is constructed following the Metropolis rule \cite[Sec. 2.4]{shi2015extra}:
\begin{align*}
    a_{ij}:=
    \begin{dcases}
        0, & (i,j)\notin\mathcal{E} \text{ and } i\neq j, \\
        \dfrac{1}{\max\limits_{i\in\mathcal{N}}\{d_i\}+1}, & (i,j)\in\mathcal{E}, \\
        1-\dfrac{d_i}{\max\limits_{i\in\mathcal{N}}\{d_i\}+1}, & i=j,
    \end{dcases}
\end{align*}
where $d_i$ is the degree of agent $i$. In all experiments, we set the network size to $N=20$.

\subsubsection{Initial values, metric \& termination criteria}
We initialize all variables as zero vectors: $\bs{s}^0 = \bs{v}^0 = \bs{x}^0 = \bs{0}_{Np}$ and $\bs{z}^0 = \bs{0}_m$, and choose $\lambda_U=1-\lambda_{\min}(W),\,\lambda_A^i=\lambda_{\max}(A_iA_i^\top)$. For fair comparison, we monitor the convergence during implementation by computing the relative residual $\eta_{\text{re}}$ defined as
\begin{align*}
    \eta_{\text{re}}:=\max&\left\{ \frac{\left\|\bar x-\prox_{\sum_{i=1}^Nr_i}(\bar x-\sum_{i=1}^NA_i^\top\nabla f_i(A_i\bar x))\right\|}{1+\|\bar x\|+\left\|\sum_{i=1}^NA_i^\top\nabla f_i(A_i\bar x)\right\|},\, \frac{\|\bs U\bs x\|}{1+\|\bs x\|} \right\},
\end{align*}
where $\bar x=\tfrac{1}{N}\sum_{i=1}^Nx_i$. The algorithms would terminate either when $\eta_{\text{re}}$ achieves the predefined accuracy threshold $\epsilon$, or when the total iteration $k$ reaches the maximum allowable number of iterations $k_{\text{max}}$. Notably, $\eta_{\text{re}}$ incorporates the consensus error $\|\bs U\bs x\|$. Numerical results validating the KKT residual bounds in Theorem~\ref{thm:complexity} are provided in Appendix \ref{app:theorykkt}.

\subsubsection{Baseline algorithms}
To demonstrate the efficiency of dHPR, we compare it with NIDS \cite{li2019nids} and PG-EXTRA \cite{shi2015pgextra}, and adopt the official implementations of NIDS and PG-EXTRA from the original open-source code provided in \cite{li2019nids}. We configured the competing algorithms according to their recommended settings: PG-EXTRA uses a step-size of $1.2/L$, and NIDS employs $1.9/L$, where $L=\max_{i\in\mc{N}}\{\|A_i^\top A_i\|\}$.

\subsubsection{The restart and $\sigma$ update strategies}
Building upon the successful restart strategies and the adaptive updating mechanism for the penalty parameter $\sigma$ developed in HPR-LP \cite{chen2024hprlp} and HPR-QP \cite{chen2025hprqp}, we similarly use a distributed restart and $\sigma$ update strategy for DOPs. 

\subsection{Distributed LASSO and Group LASSO}

Consider a network of $N$ agents that aims to solve the following regularized linear regression problem:
$$
\min_{x\in\mb R^p}\quad \sum_{i=1}^N\frac{1}{2}\|A_ix-b_i\|^2+r_i(x),
$$
where $A_i\in\mb R^{m_i\times p},\,b_i\in\mb R^{m_i}$, and $r_i:\,\mb R^p\rightarrow\mb R$ are given regularizers. We consider two types of regularizers $r_i$ (LASSO and group LASSO) and test the algorithms on both synthetic and real datasets. The detailed settings are as follows:
\begin{itemize}
    \item \textbf{Synthetic data:} We set $x_{\text{true}} = \mathbf{1}_p$ and generate the measurement vectors according to $b_i = A_i x_{\text{true}} + \delta e_i$, where each entry of the data matrix $A_i \in \mb{R}^{m_i \times p}$ and noise vector $e_i \in \mb{R}^{m_i}$ are independently sampled from the standard normal distribution $\mc N(0,1)$. The noise level parameter $\delta$ is fixed at $10^{-2}$.
    
    \item \textbf{Real dataset:} We also test the algorithms on three datasets $\mathtt{abalone},\,\mathtt{mg}$ and $\mathtt{pyrim}$ from the UCI machine learning repository\footnote{\href{https://www.csie.ntu.edu.tw/~cjlin/libsvmtools/datasets/regression.html}{https://www.csie.ntu.edu.tw/~cjlin/libsvmtools/datasets/regression.html}}, and the training samples are randomly and evenly distributed over all the $N$ agents.

    \item\textbf{LASSO:} The regularization term for each agent $i$ is formulated as:
    $$
    r_i(x) = \theta_i\|x\|_1,
    $$
    where $\theta_i = 0.01\|A_i^\top b_i\|_\infty$.

    \item\textbf{Group LASSO:} The group LASSO regularization term for each agent $i$ is defined as:
    $$
    r_i(x) = \theta_{1,i}\|x\|_1 + \theta_{2,i}\sum_{l=1}^g w_l\|x_{G_l}\|_2,
    $$
    where the hyper-parameters are set as $\theta_{1,i} = \theta_{2,i} = 0.01\|A_i^\top b_i\|_{\infty}$. Here, the features are partitioned into $g$ adjacent index groups $\{G_l\}_{l=1}^g$, where each group $G_l \subseteq \{1,\,\ldots,\,p\}$ represents a distinct feature subset. The group weights are assigned as $w_l = \sqrt{|G_l|}$ to account for varying group sizes \cite{zhang2020grouplasso}. The restriction of $x$ to group $G_l$ is denoted by $x_{G_l}$. The group sizes $\{|G_l|\}_{l=1}^g$ are randomly generated such that their expected value approximates $p/g$. In our experiments, we set $g = p/10$ by default.
    \item \textbf{Stopping Criteria:} We stop all algorithms if $\eta_{re} < 10^{-8}$ or the maximum iteration number $k_\text{max} = 20000$ is reached.
\end{itemize}

For the regularized linear regression problems, we construct the random communication graph with $\iota=0.5$. Firstly, we test the scalability of the algorithms using three synthetic datasets with increasing dimensions: $(m_i,\,p)\in\{(10,\,50),\,(100,\,500),\,(1000,\,5000)\}$. The results are shown in Fig. \ref{subfig:lasso_synthetic} and Fig. \ref{subfig:glasso_synthetic} for the LASSO regression and group LASSO regression models, respectively. The results demonstrate that dHPR is scalable and consistently outperforms NIDS and PG-EXTRA on the synthetic data.  

We further evaluate the performance of the algorithms on three UCI datasets: $\mathtt{abalone},\,\mathtt{mg}$ and $\mathtt{pyrim}$. The results are shown in Fig. \ref{subfig:lasso_uci} and Fig. \ref{subfig:glasso_uci} for the LASSO regression and group LASSO regression models, respectively. These results demonstrate that dHPR outperforms NIDS and PG-EXTRA on the real data as well. 

To investigate the impact of network topology on algorithm performance, we evaluate dHPR across three different communication graphs: line, random graph, and complete graph, where the random graph has a connectivity ratio of $\iota=0.2$.
The experiments are conducted using a synthetic dataset with $(m_i,p)=(100,500)$ for demonstration purposes, while other parameters are kept unchanged as in previous experiments. As shown in Figs. \ref{subfig:topology_lasso} and \ref{subfig:topology_lgasso}, dHPR exhibits faster convergence in more densely connected networks (convergence rate: complete graph $>$ random graph $>$ line), which is consistent with previous findings (e.g., \cite{shi2014dadmm,ling2015dlm,li2022tpus,aybat2017pgadmm}). 
This phenomenon further demonstrates the robustness of dHPR under varying communication conditions.

\begin{figure*}[htbp]
\centering
\begin{subfigure}{\linewidth}
	\centering
	\begin{subfigure}{0.3\linewidth}
	\includegraphics[width=\linewidth]{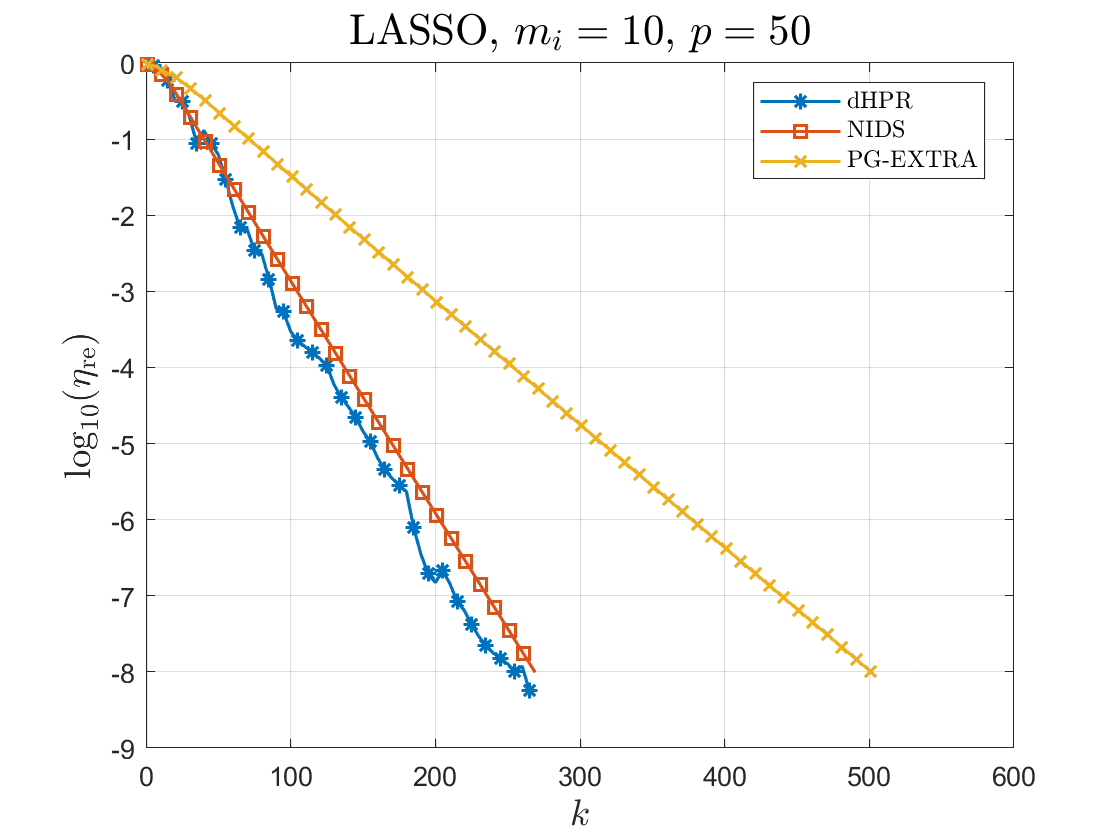}
	\end{subfigure}
	\hfill
	\begin{subfigure}{0.3\linewidth}
	\includegraphics[width=\linewidth]{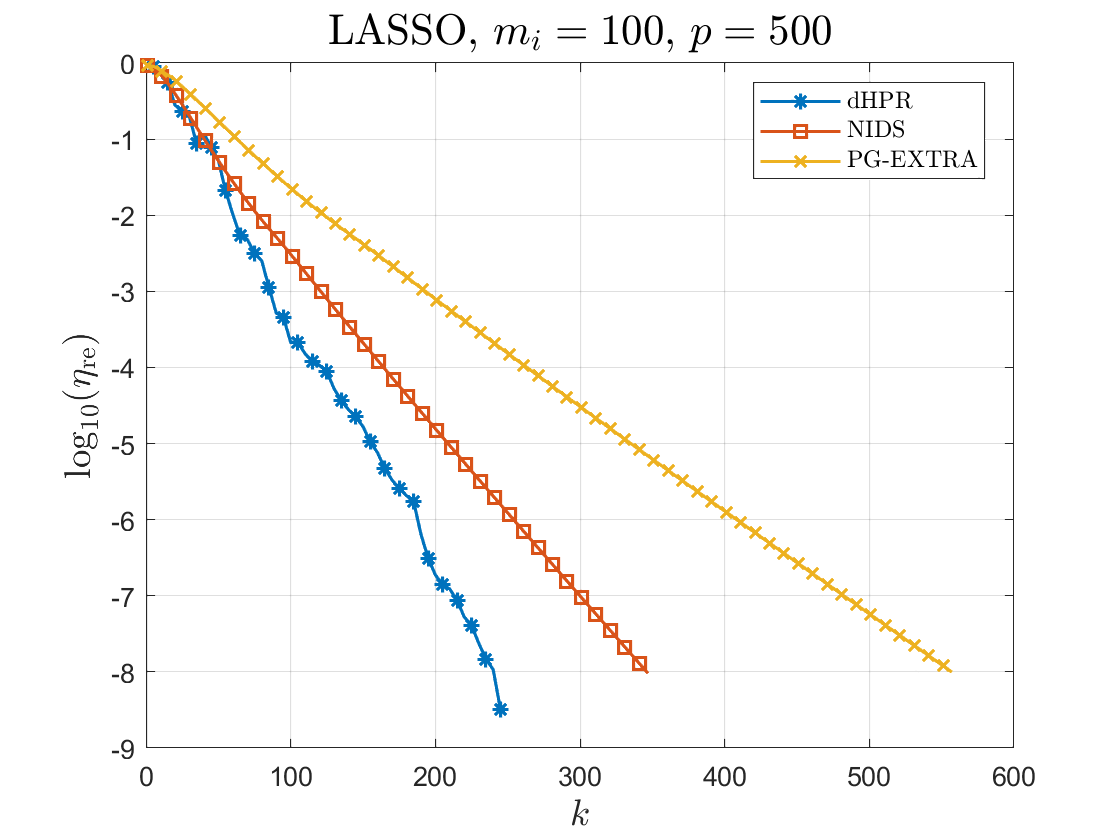}
	\end{subfigure}
	\hfill
	\begin{subfigure}{0.3\linewidth}
	\includegraphics[width=\linewidth]{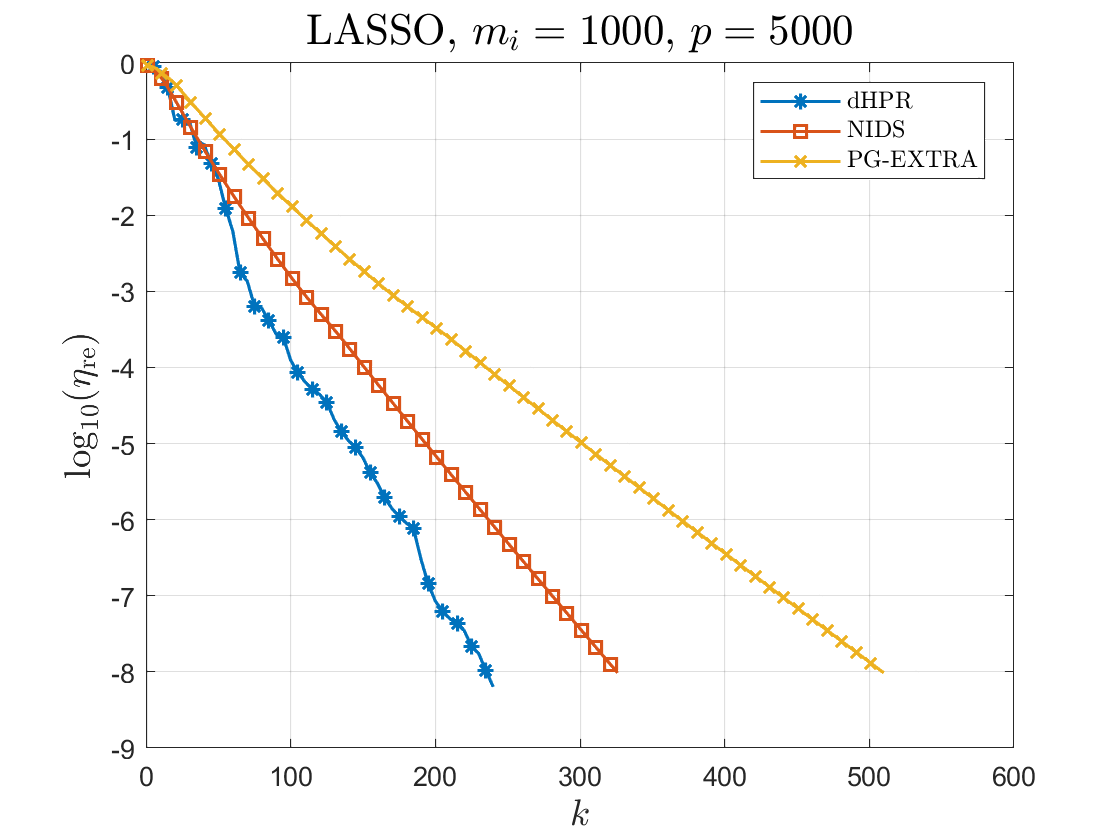}
	\end{subfigure}
	\caption{Synthetic data}
    \label{subfig:lasso_synthetic}
\end{subfigure}

\vspace{1em}
\begin{subfigure}{\linewidth}
	\centering
	\begin{subfigure}{0.3\linewidth}
	\includegraphics[width=\linewidth]{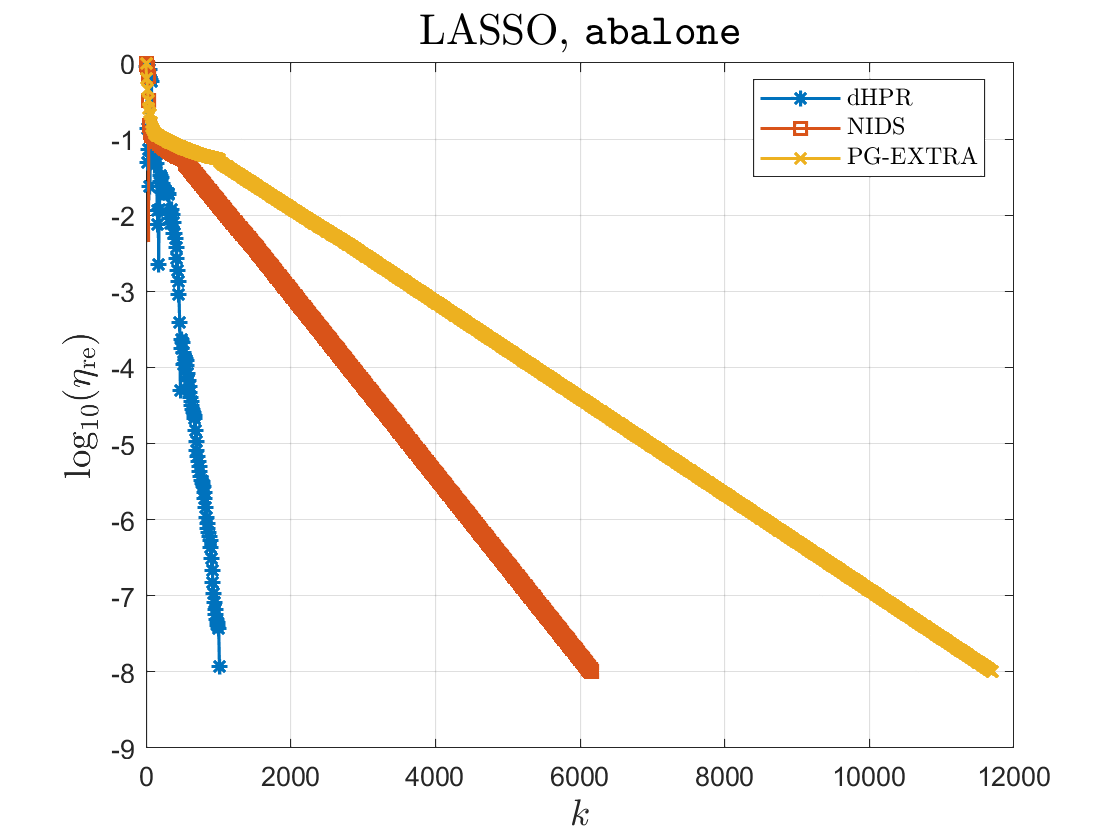}
	\end{subfigure}
	\hfill
	\begin{subfigure}{0.3\linewidth}
	\includegraphics[width=\linewidth]{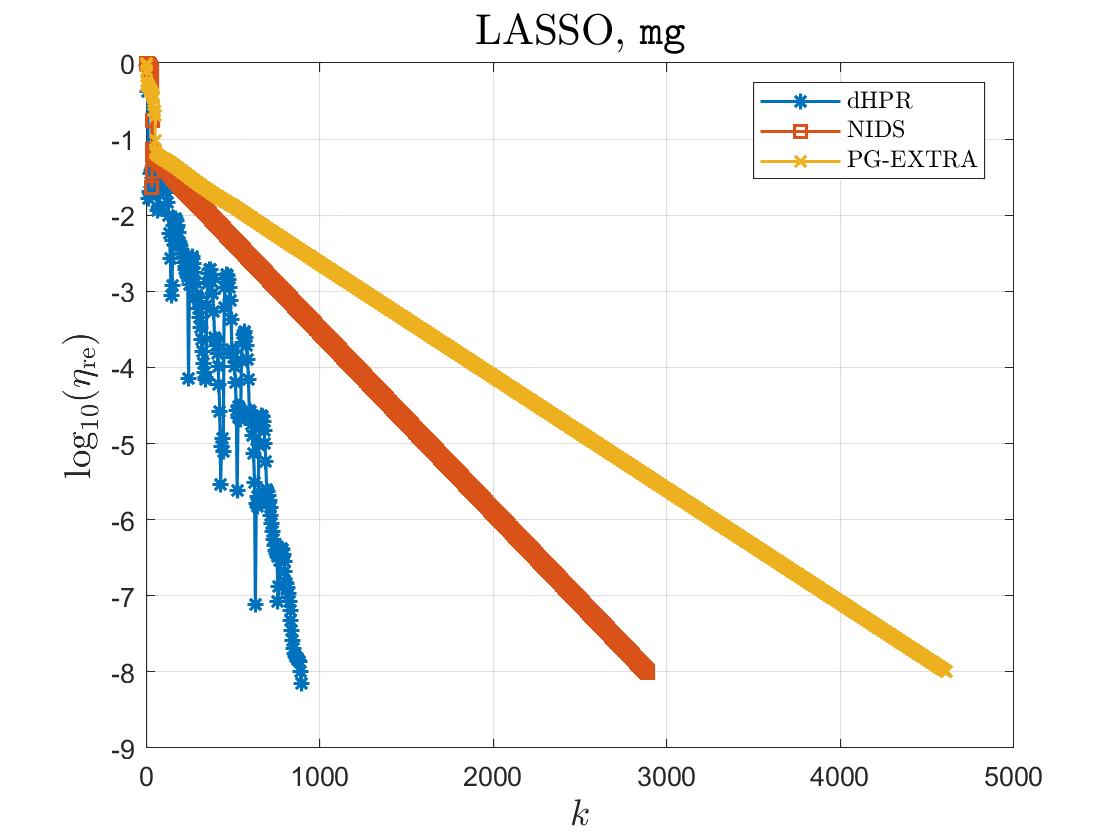}
	\end{subfigure}
	\hfill
	\begin{subfigure}{0.3\linewidth}
	\includegraphics[width=\linewidth]{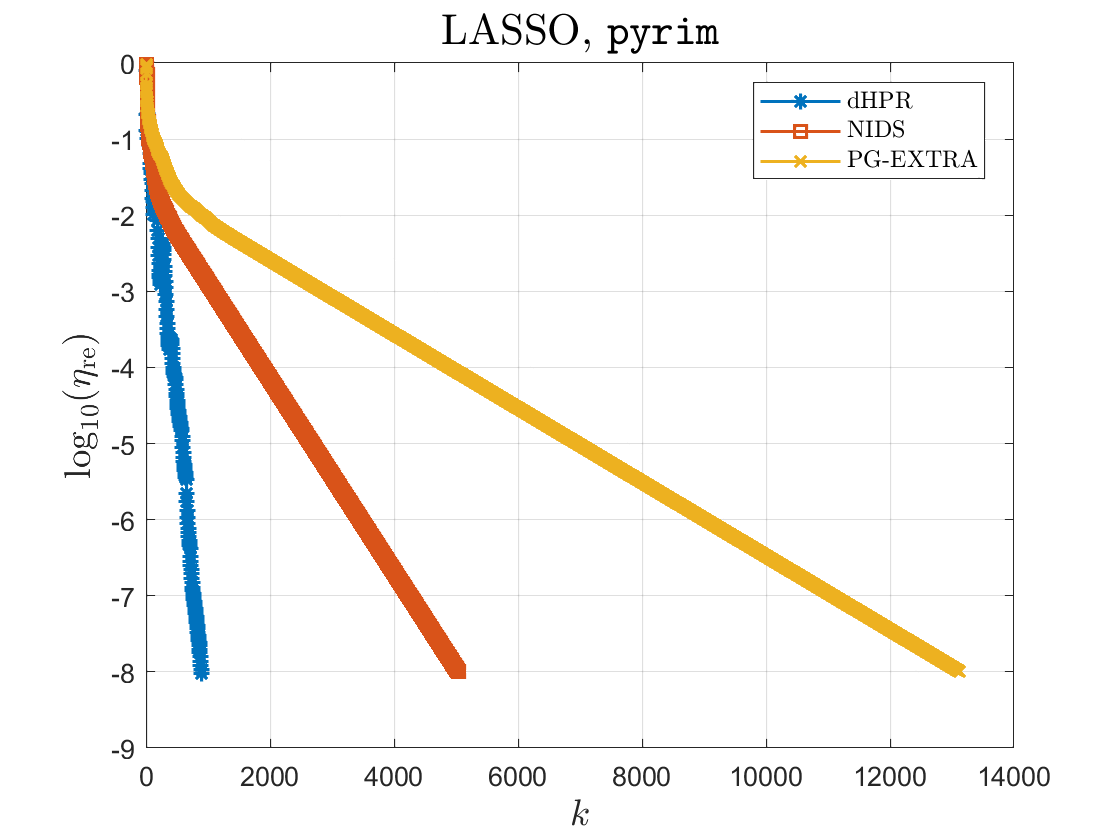}
	\end{subfigure}
	\caption{UCI instances}
    \label{subfig:lasso_uci}
\end{subfigure}
\caption{Distributed LASSO}
\label{fig:lasso}
\end{figure*}

\begin{figure*}[htbp]
\centering
\begin{subfigure}{\linewidth}
	\centering
	\begin{subfigure}{0.3\linewidth}
	\includegraphics[width=\linewidth]{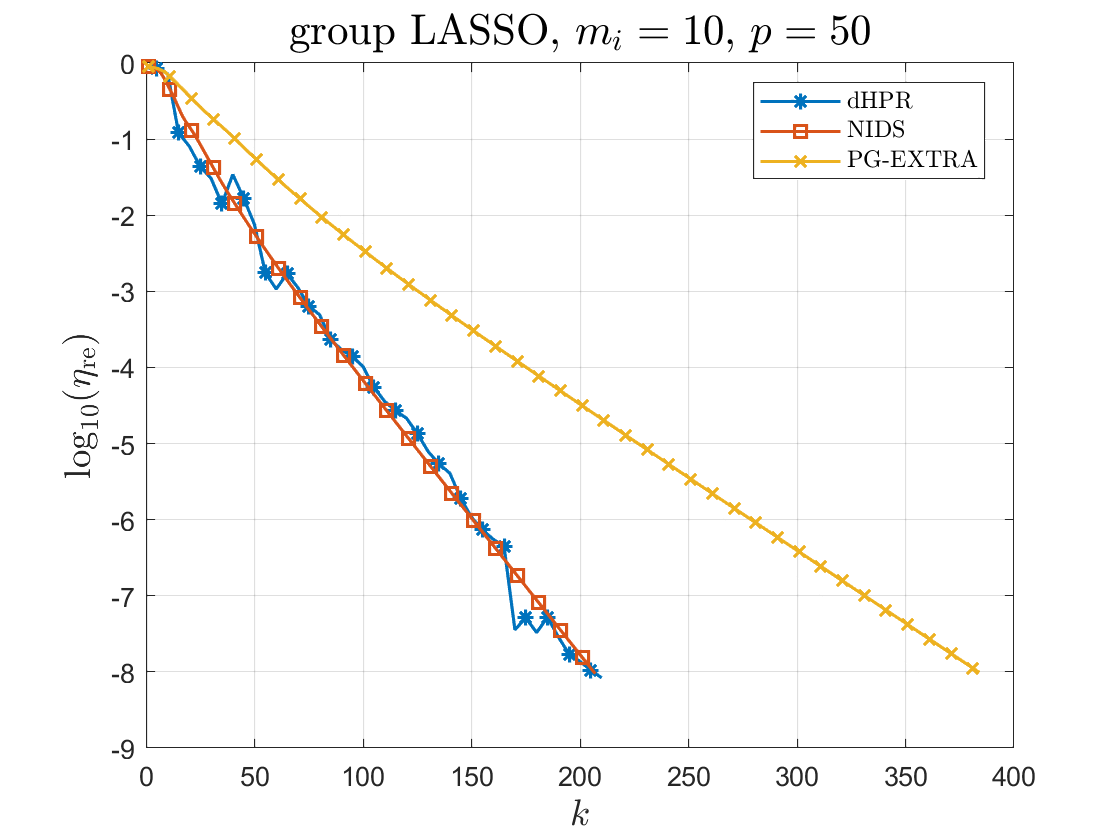}
	\end{subfigure}
	\hfill
	\begin{subfigure}{0.3\linewidth}
	\includegraphics[width=\linewidth]{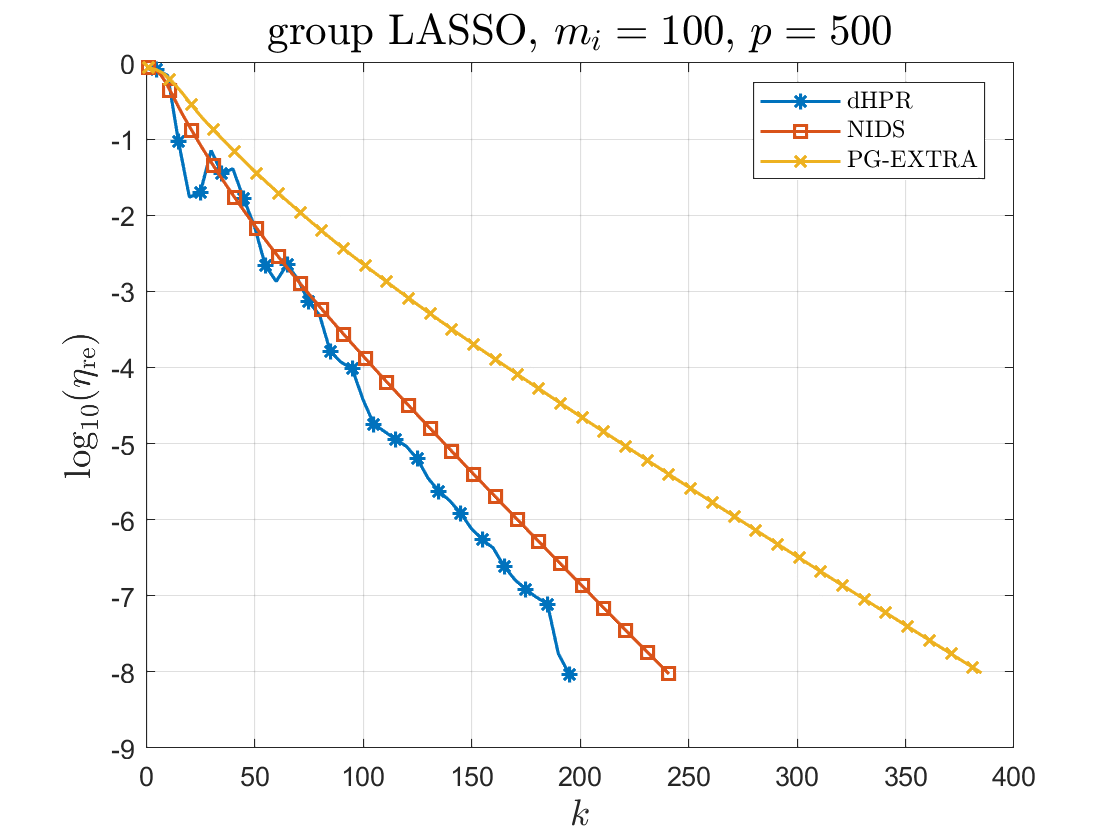}
	\end{subfigure}
	\hfill
	\begin{subfigure}{0.3\linewidth}
	\includegraphics[width=\linewidth]{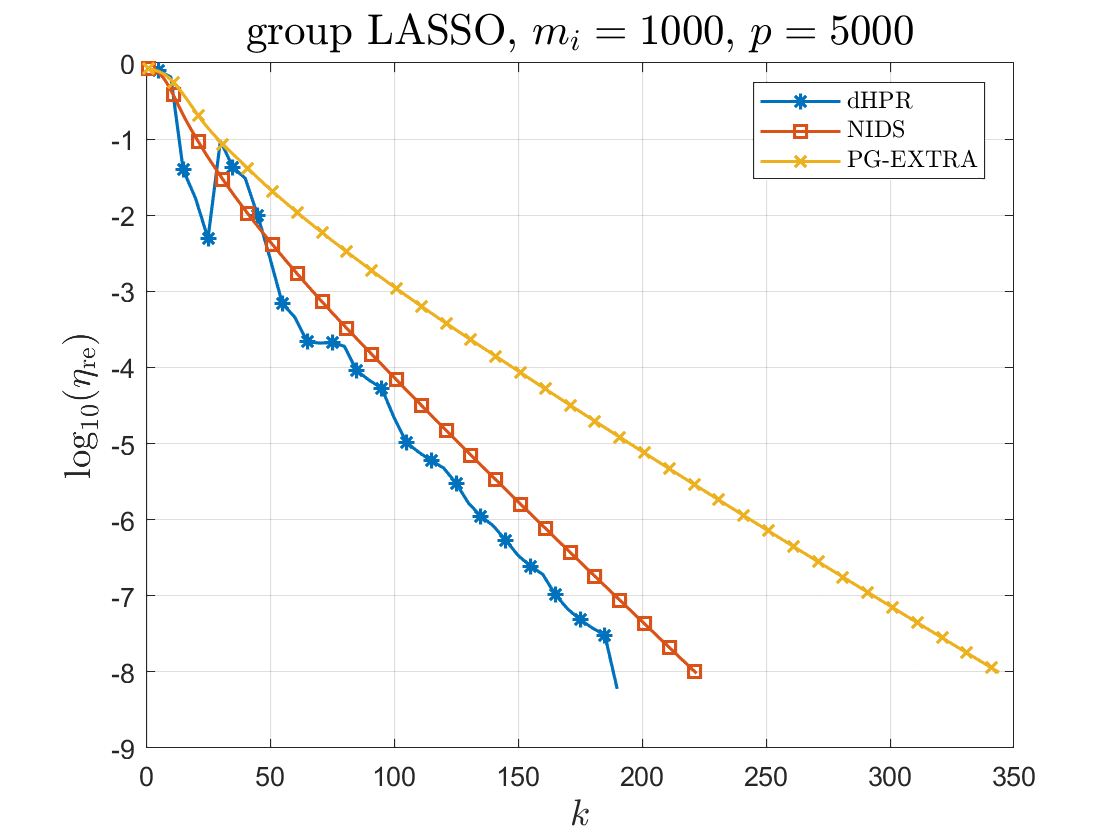}
	\end{subfigure}
	\caption{Synthetic data}
    \label{subfig:glasso_synthetic}
\end{subfigure}

\vspace{1em}
\begin{subfigure}{\linewidth}
	\centering
	\begin{subfigure}{0.3\linewidth}
	\includegraphics[width=\linewidth]{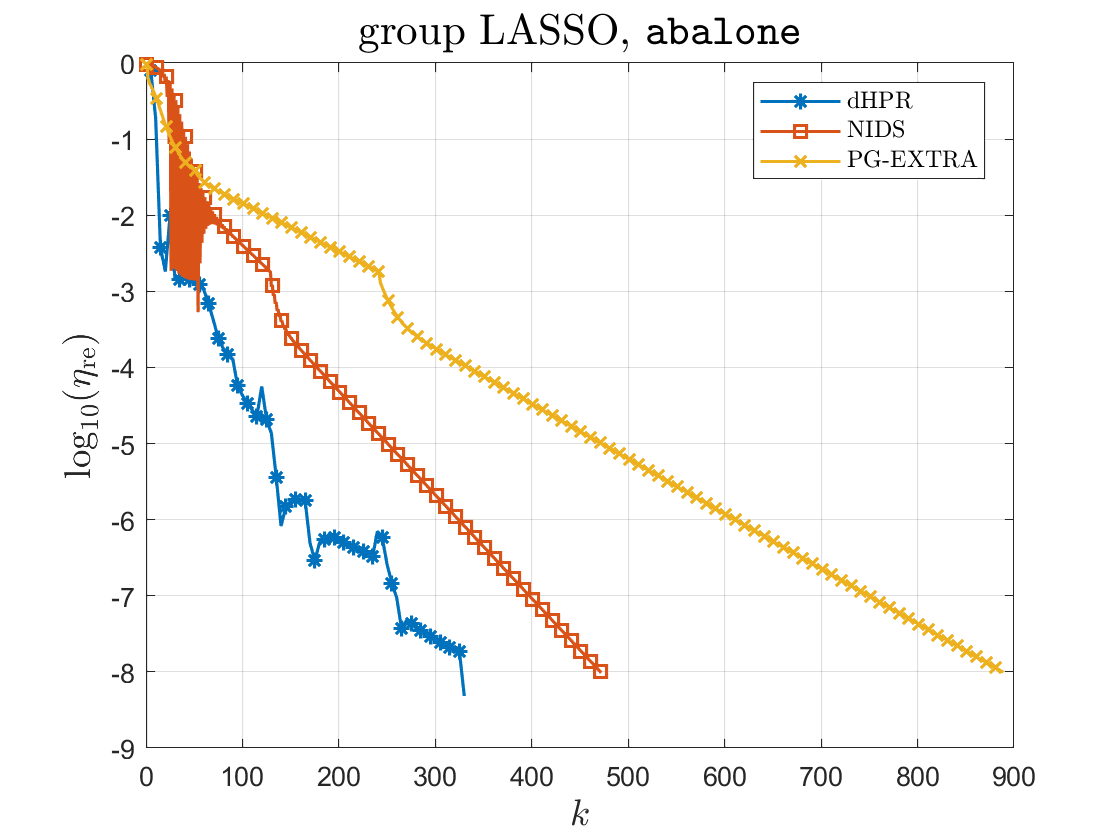}
	\end{subfigure}
	\hfill
	\begin{subfigure}{0.3\linewidth}
	\includegraphics[width=\linewidth]{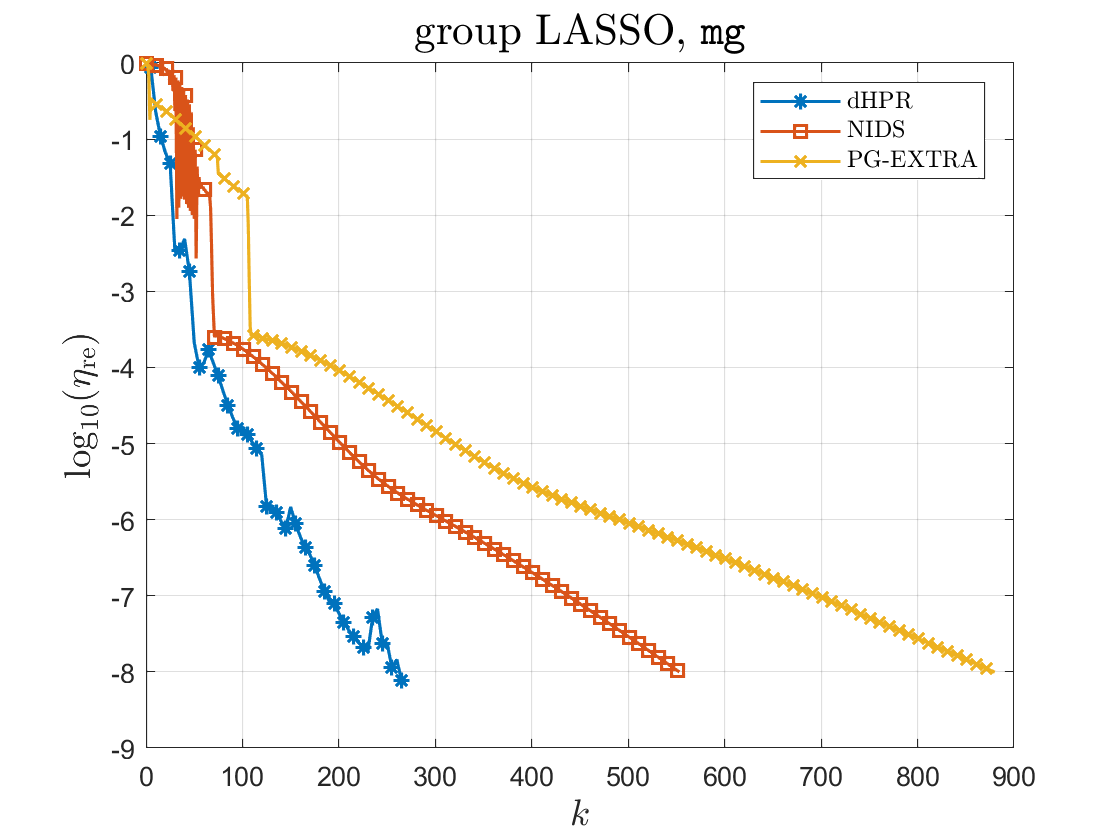}
	\end{subfigure}
	\hfill
	\begin{subfigure}{0.3\linewidth}
	\includegraphics[width=\linewidth]{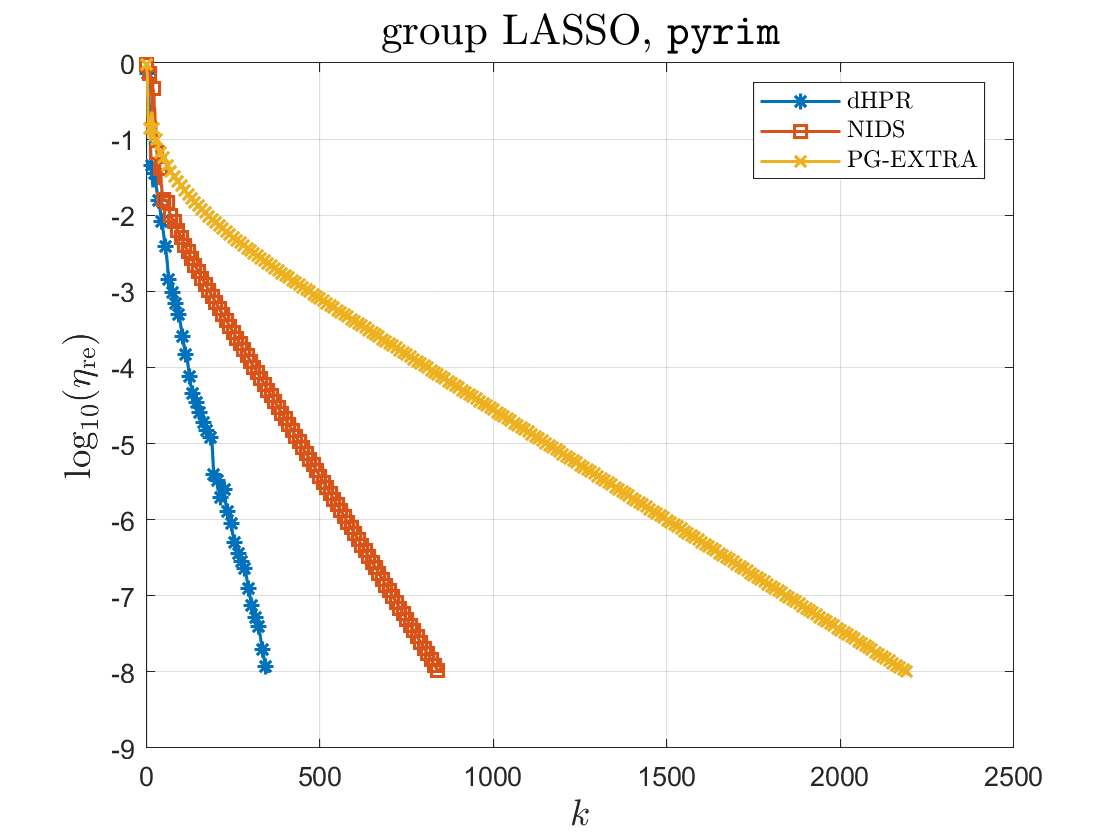}
	\end{subfigure}
	\caption{UCI instances}
    \label{subfig:glasso_uci}
\end{subfigure}
\caption{Distributed group LASSO}
\label{fig:glasso}
\end{figure*}

\subsection{Distributed $L_1$-Regularized Logistic Regression}

Consider a network consisting of $N$ agents collaboratively solving the following problem:
\[
\min_{x \in \mathbb{R}^p}\quad \sum_{i=1}^N \left( \sum_{l=1}^{m_i} \ln[1 + \exp(-b_{il}\bs{a}_{il}^\top x)] + \theta_i \|x\|_1 \right),
\]
where each agent $i$ maintains a private training dataset consisting of $m_i$ samples $\{(\bs{a}_{il},\,b_{il})\}_{l=1}^{m_i}$, with $\bs{a}_{il}\in\mb{R}^p$ representing the feature vector and $b_{il}\in\{1,\,-1\}$ denoting the binary label. 
The detailed settings are as follows:
\begin{itemize}
    \item \textbf{Synthetic data:} Feature vectors $\bs{a}_{il}$ corresponding to positive labels ($b_{il}=1$) are drawn from the normal distribution $\mc{N}(0.1,1)$, and those with negative labels ($b_{il}=-1$) are sampled from $\mc{N}(-0.1,1)$.
    \item\textbf{Real dataset:} We select 23 UCI instances (see Table \ref{tab:logi}), and the training samples are randomly and evenly distributed over all the $N$ agents.
    \item\textbf{Regularizer:} The regularization parameter $\theta_i=0.01\|A_i^\top b_i\|_\infty$ for all $i\in\mc N$.
\end{itemize}

Firstly, we construct a random communication graph with $\iota=0.5$, set $\epsilon=10^{-8},\,k_\text{max}=10000$, and compare dHPR with NIDS and PG-EXTRA using synthetic datasets with increasing dimensions: $(m_i,p)\in\{(10,50),(100,500),(500,1000)\}$. The results are shown in Fig. \ref{subfig:logi_synthetic}, which demonstrates the scalability of dHPR and the superiority of dHPR over NIDS and PG-EXTRA on the synthetic data.

Then, we evaluate the algorithms in 23 UCI instances over the same communication graph, and report the number of iterations required to reach the predetermined accuracies $\epsilon=10^{-4},\,10^{-6}$ and $10^{-8}$ for each algorithm in Table \ref{tab:logi}. For clarity of comparison, the curves of $\log_{10}(\eta_{\text{re}})$ for instances $\mathtt{a9a}$, $\mathtt{ijcnn1}$ and $\mathtt{w3a}$ are plotted for the first $10000$ iterations and shown in Fig. \ref{subfig:logi_uci}. The results demonstrate that dHPR exhibits superior convergence performance over the competing methods, as NIDS and PG-EXTRA either require more iterations or fail to reach the same accuracy within $k_{\text{max}}=50000$.

We also compare the performance of dHPR for solving this problem over different topologies, following the same experimental setup as in LASSO and group LASSO. The results, presented in Fig. \ref{subfig:topology_logi}, demonstrate that dHPR is robust and achieves faster convergence on topologies with higher connectivity, which is consistent with the trends observed in the previous experiments. These results illustrate the effectiveness of dHPR in solving distributed $L_1$-regularized logistic regression problems.

\begin{figure*}[htbp]
\centering
\begin{subfigure}{\linewidth}
	\centering
	\begin{subfigure}{0.3\linewidth}
	\includegraphics[width=\linewidth]{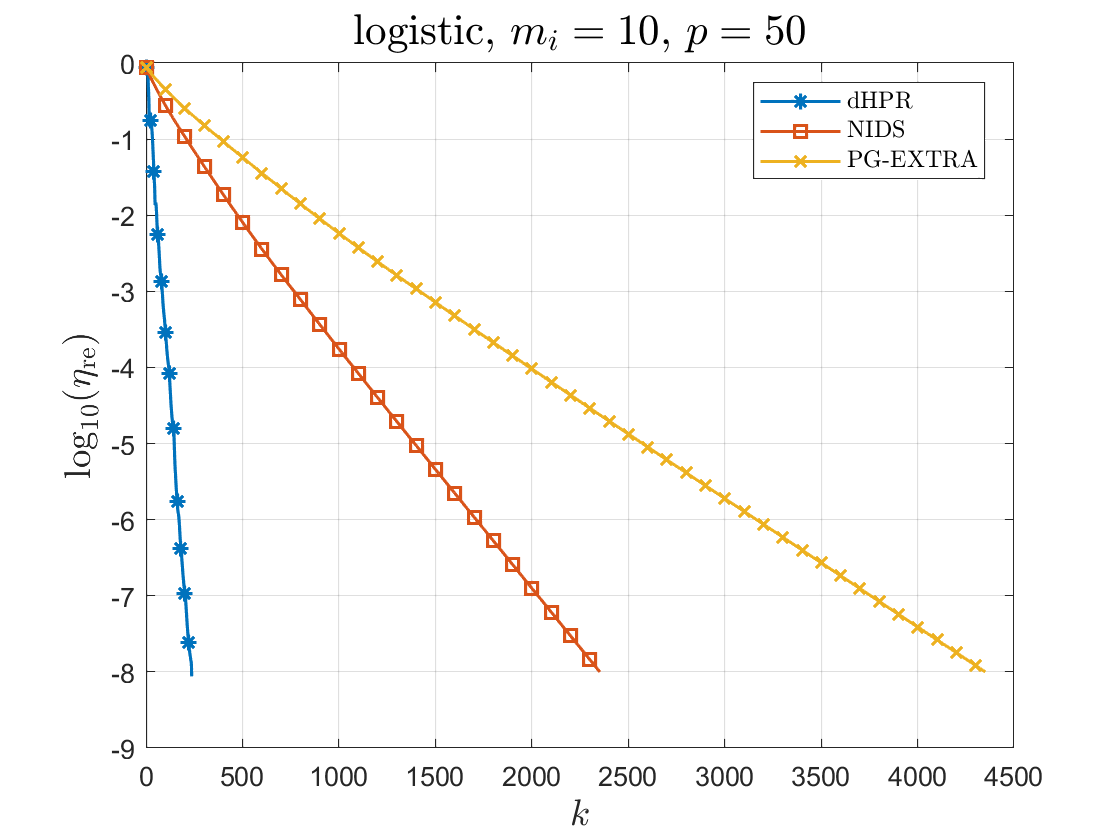}
	\end{subfigure}
	\hfill
	\begin{subfigure}{0.3\linewidth}
	\includegraphics[width=\linewidth]{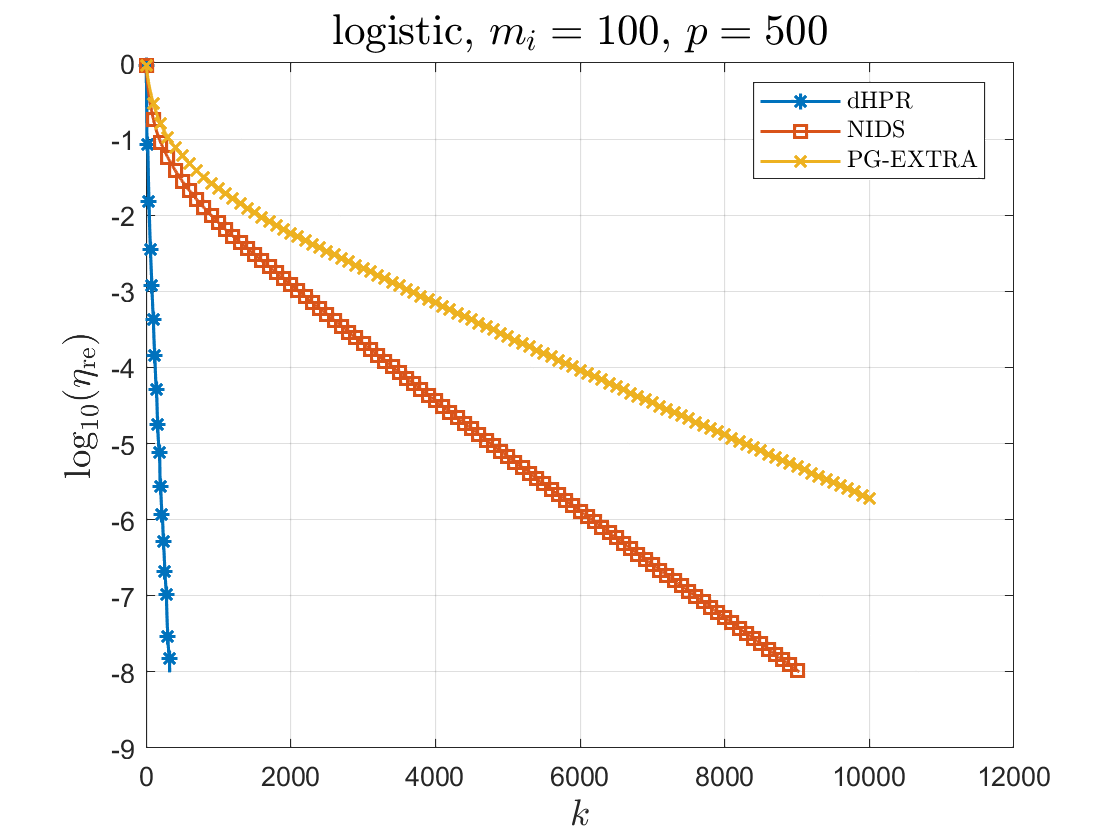}
	\end{subfigure}
	\hfill
	\begin{subfigure}{0.3\linewidth}
	\includegraphics[width=\linewidth]{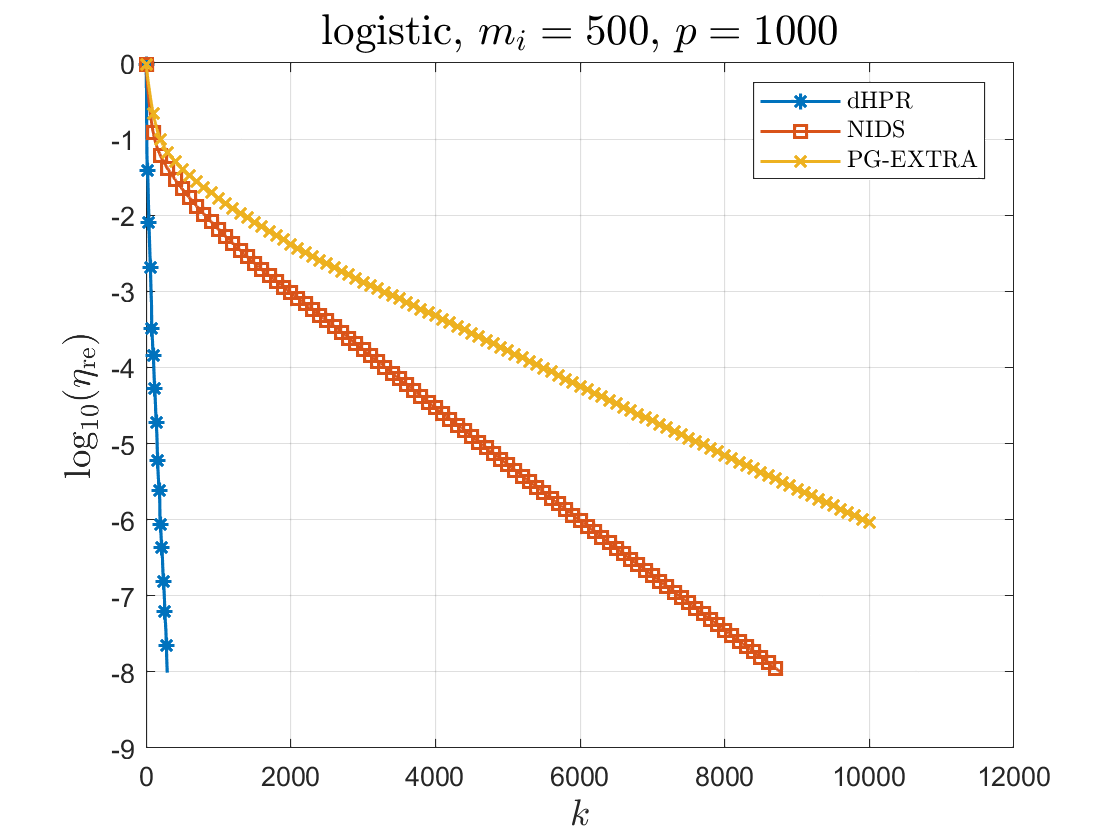}
	\end{subfigure}
	\caption{Synthetic data}
    \label{subfig:logi_synthetic}
\end{subfigure}

\vspace{1em}
\begin{subfigure}{\linewidth}
	\centering
	\begin{subfigure}{0.3\linewidth}
	\includegraphics[width=\linewidth]{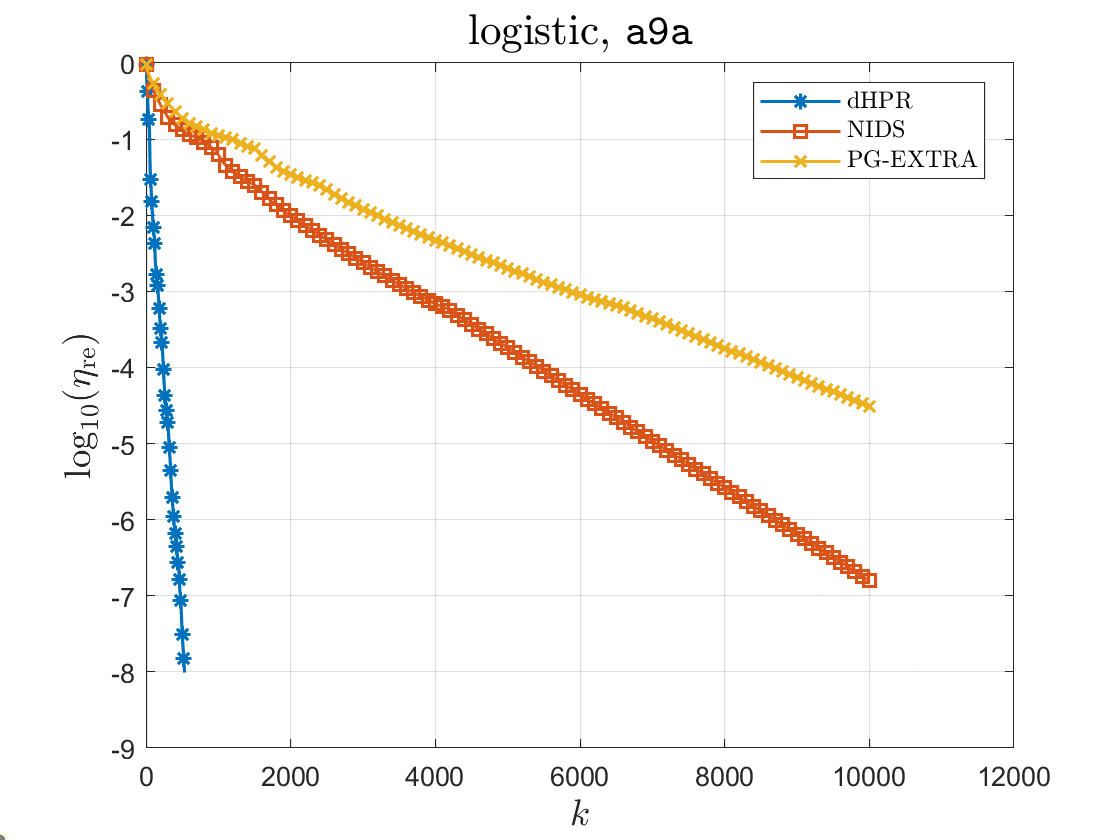}
	\end{subfigure}
	\hfill
	\begin{subfigure}{0.3\linewidth}
	\includegraphics[width=\linewidth]{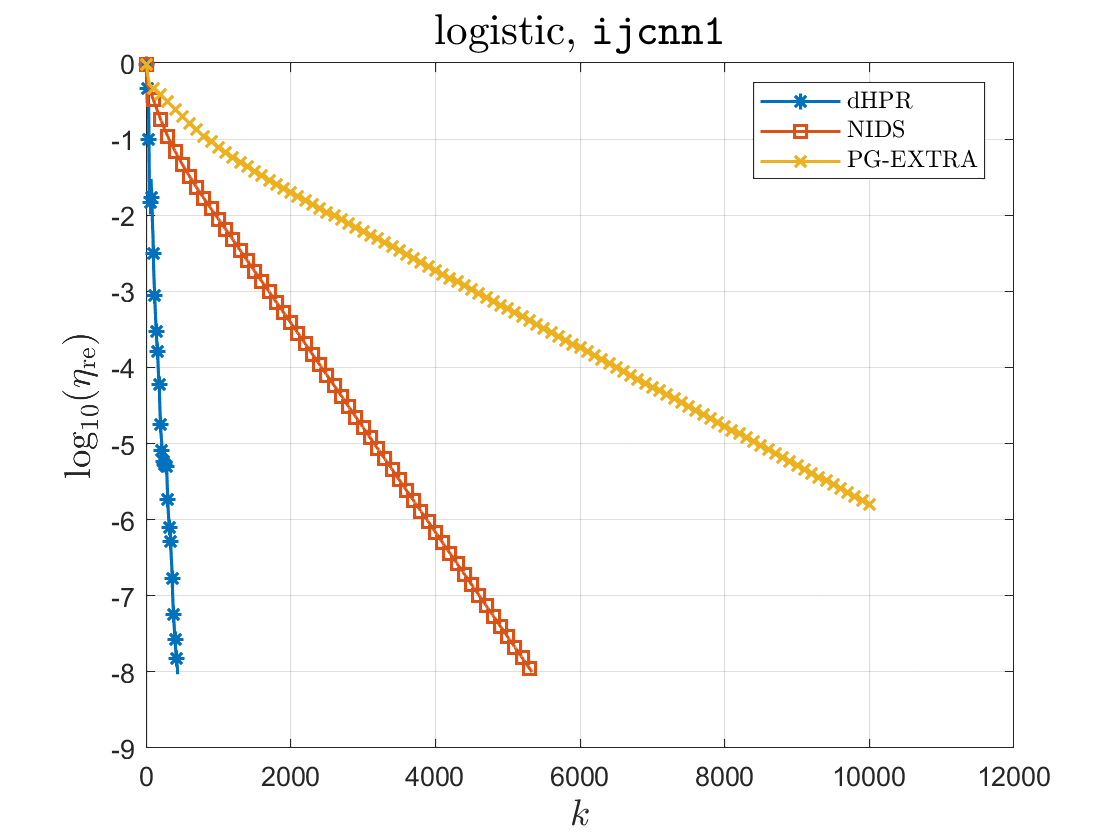}
	\end{subfigure}
	\hfill
	\begin{subfigure}{0.3\linewidth}
	\includegraphics[width=\linewidth]{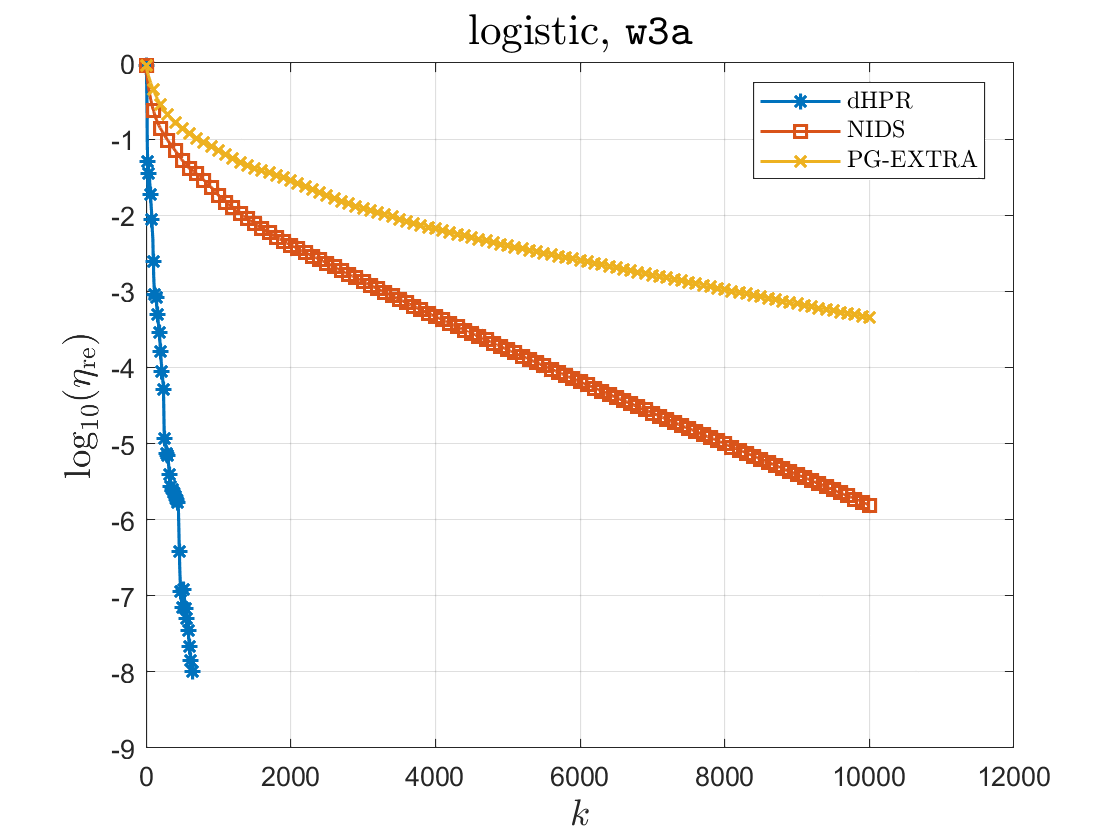}
	\end{subfigure}
	\caption{UCI instances}
    \label{subfig:logi_uci}
\end{subfigure}
\caption{Distributed $L_1$-regularized logistic regression}
\label{fig:logi}
\vspace{1em}
\resizebox{\textwidth}{!}{%
\begin{tabular}{c>{\bfseries}cccc>{\bfseries}cccc>{\bfseries}ccc}
\toprule
\multirow{2}{*}{Instance} & \multicolumn{3}{c}{$\epsilon=10^{-4}$} &  & \multicolumn{3}{c}{$\epsilon=10^{-6}$} &  & \multicolumn{3}{c}{$\epsilon=10^{-8}$} \\ 
\cline{2-4} \cline{6-8} \cline{10-12} 
\addlinespace
 & dHPR & NIDS & PG-EXTRA &  & dHPR & NIDS & PG-EXTRA &  & dHPR & NIDS & PG-EXTRA \\ 
\midrule
$\mathtt{a1a}$                       & 237     & 5333  & 9014     &  & 400     & 8636  & 14598    &  & 552     & 12110 & 20470    \\
$\mathtt{a2a}$                       & 237     & 5224  & 8640     &  & 372     & 8391  & 13877    &  & 512     & 11560 & 19117    \\
$\mathtt{a3a}$                       & 241     & 5255  & 8692     &  & 380     & 8414  & 13918    &  & 533     & 11575 & 19147    \\
$\mathtt{a4a}$                       & 255     & 5501  & 9072     &  & 418     & 8763  & 14452    &  & 564     & 12025 & 19834    \\
$\mathtt{a5a}$                       & 231     & 5798  & 9387     &  & 372     & 9064  & 14675    &  & 507     & 12345 & 19982    \\
$\mathtt{a6a}$                       & 237     & 5377  & 8639     &  & 386     & 8641  & 13881    &  & 557     & 11906 & 19125    \\
$\mathtt{a7a}$                       & 230     & 5474  & 8868     &  & 366     & 8758  & 14190    &  & 509     & 12042 & 19514    \\
$\mathtt{a8a}$                       & 232     & 5421  & 8763     &  & 369     & 8697  & 14016    &  & 511     & 11974 & 19299    \\
$\mathtt{a9a}$                       & 240     & 5419  & 8660     &  & 385     & 8686  & 13881    &  & 533     & 11957 & 19106    \\
$\mathtt{australian}$                & 3509    & 8607  & F        &  & 3516    & 10189 & F        &  & 7241    & 33317 & F       \\
$\mathtt{colon}$-$\mathtt{cancer}$   & 1413    & F     & F        &  & 2585    & F     & F        &  & 3841    & F     & F       \\
$\mathtt{diabetes}$                  & 501     & 2249  & 5686     &  & 719     & 3256  & 8232     &  & 909     & 4263  & 10780    \\
$\mathtt{heart}$                     & 807     & 7740  & 14936    &  & 1439    & 11338 & 21879    &  & 1808    & 14924 & 28853    \\
$\mathtt{ijcnn1}$                    & 169     & 2426  & 6498     &  & 315     & 3879  & 10390    &  & 437     & 5333  & 14282    \\
$\mathtt{svmguide3}$                 & 345     & 3698  & 14098    &  & 521     & 5938  & 22647    &  & 725     & 8178  & 31199    \\
$\mathtt{w1a}$                       & 219     & 5713  & 14642    &  & 448     & 10970 & 28110    &  & 710     & 16282 & 41723    \\
$\mathtt{w2a}$                       & 194     & 10676 & 26409    &  & 445     & 20960 & F        &  & 679     & 31241 & F       \\
$\mathtt{w3a}$                       & 217     & 5549  & 13765    &  & 452     & 10478 & 25988    &  & 642     & 15503 & 38455    \\
$\mathtt{w4a}$                       & 249     & 5528  & 14380    &  & 475     & 10713 & 27866    &  & 795     & 16058 & 41773    \\
$\mathtt{w5a}$                       & 212     & 5283  & 14039    &  & 443     & 10197 & 27097    &  & 703     & 15377 & 40865    \\
$\mathtt{w6a}$                       & 210     & 5387  & 13240    &  & 365     & 9949  & 24450    &  & 526     & 14959 & 36762    \\
$\mathtt{w7a}$                       & 268     & 5291  & 12653    &  & 486     & 10265 & 24547    &  & 773     & 15360 & 36731    \\
$\mathtt{w8a}$                       & 220     & 5423  & 13931    &  & 391     & 10365 & 26629    &  & 554     & 15434 & 39655    \\
\bottomrule
\end{tabular}%
}
\captionof{table}{\small The iteration number of achieving the predetermined accuracy $\epsilon$ in distributed $L_1$-regularized logistic regression problems (`F' = Failure, i.e., the algorithm did not reach the required accuracy within the maximum iteration number).}
\label{tab:logi}
\end{figure*}

\begin{figure*}[htbp]
\centering
    \begin{subfigure}{0.3\linewidth}
        \includegraphics[width=\linewidth]{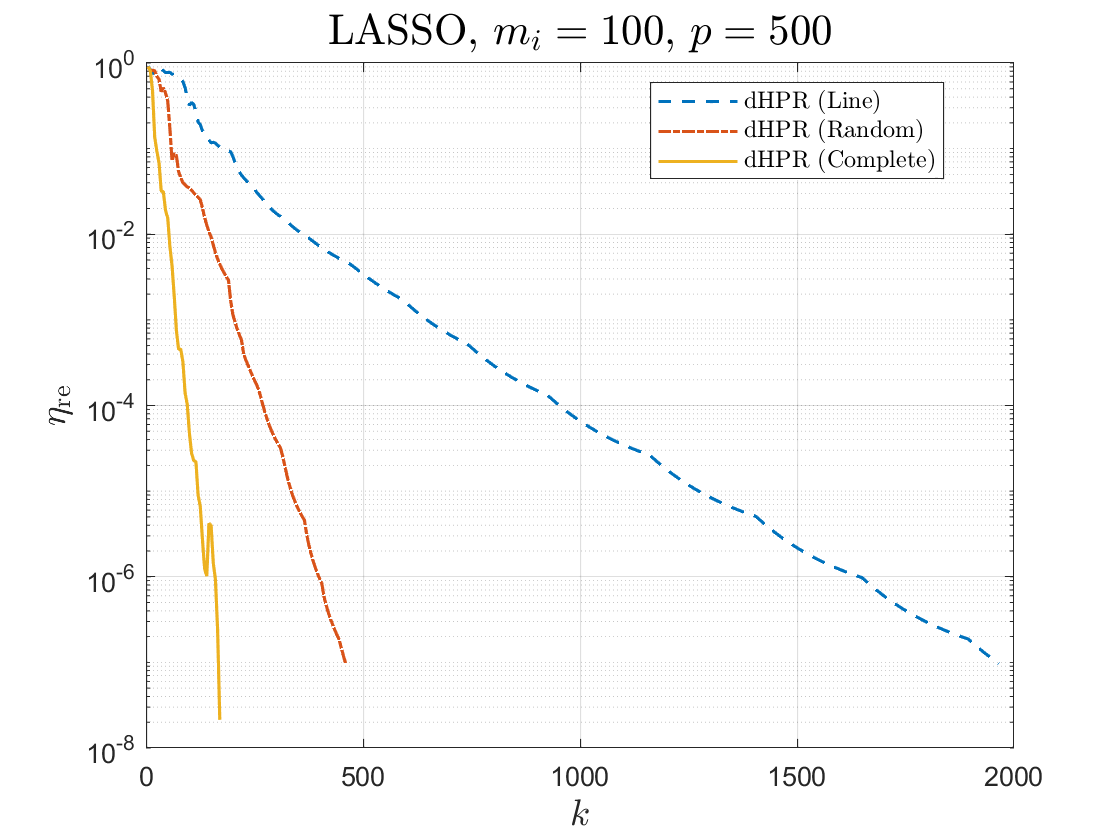}
        \caption{LASSO}
        \label{subfig:topology_lasso}
    \end{subfigure}
    \hfill
    \begin{subfigure}{0.3\linewidth}
        \includegraphics[width=\linewidth]{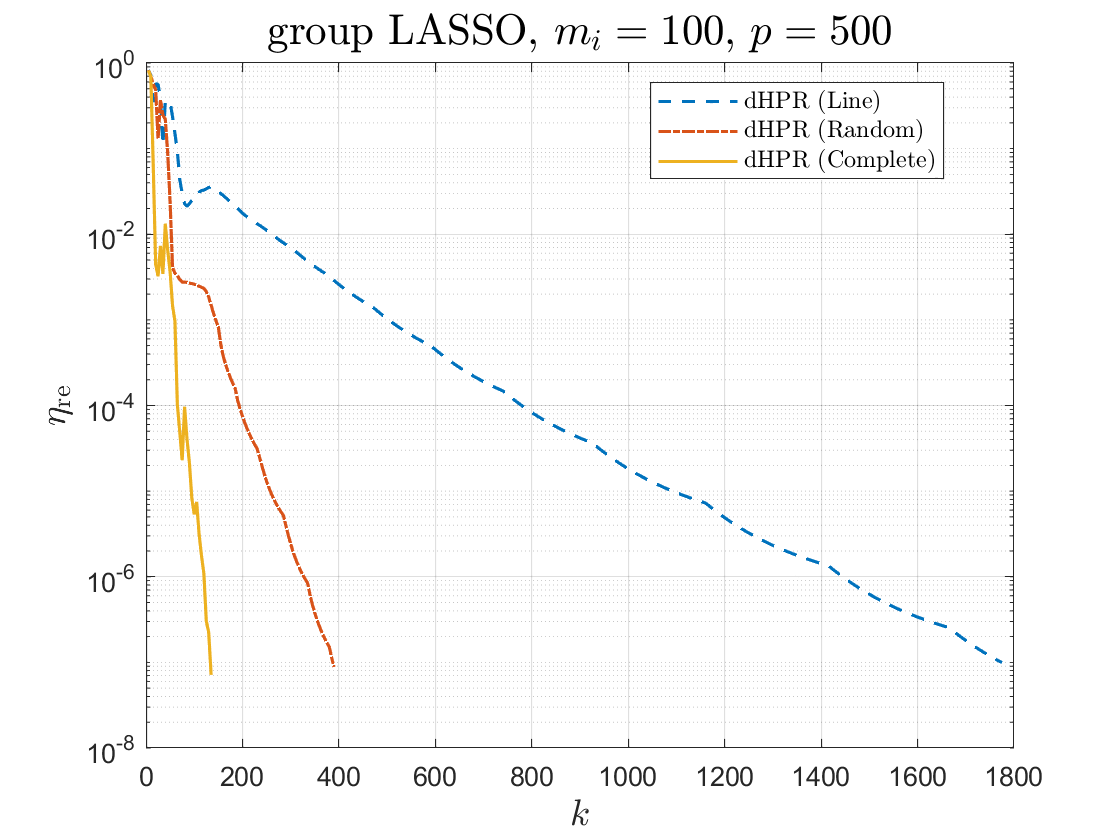}
        \caption{Group LASSO}
        \label{subfig:topology_lgasso}
    \end{subfigure}
    \hfill
    \begin{subfigure}{0.3\linewidth}
        \includegraphics[width=\linewidth]{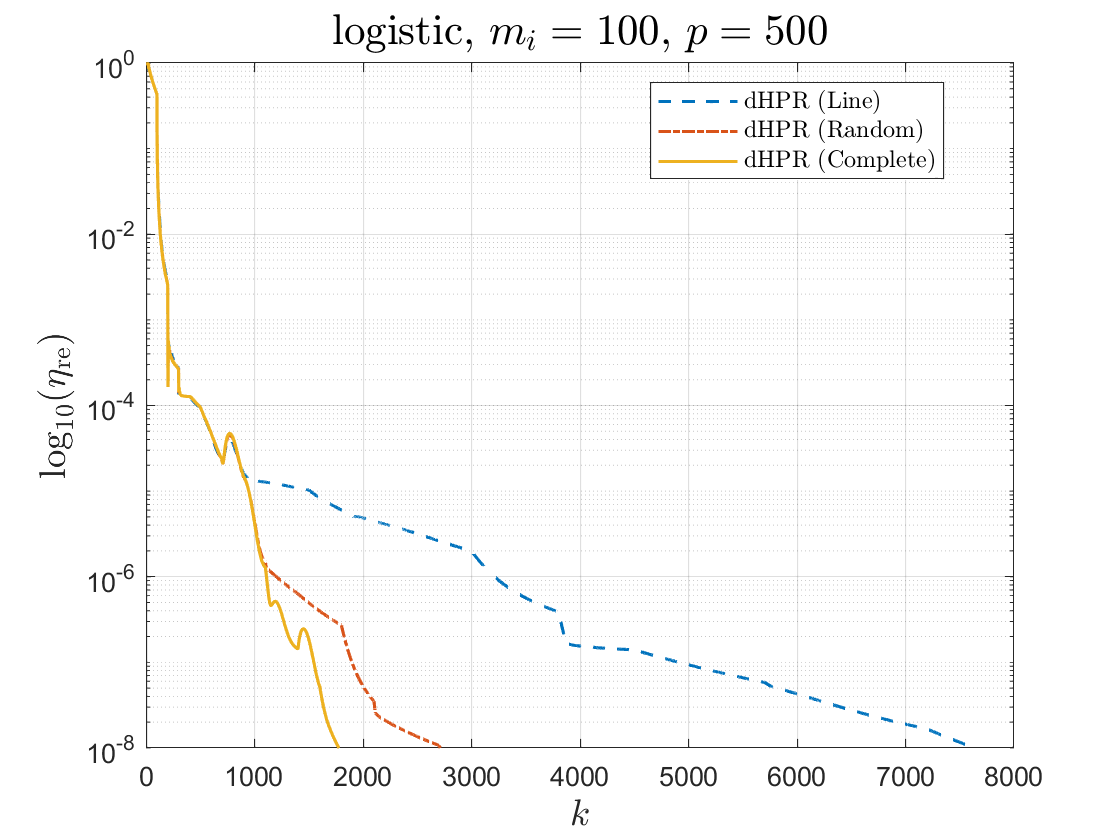}
        \caption{Logistic regression}
        \label{subfig:topology_logi}
    \end{subfigure}
\caption{Effect of topologies}
\label{fig:topology}
\end{figure*}

\section{Conclusion}\label{sec:conclusion}

This paper introduced the dHPR method for solving composite distributed optimization problems with non-smooth objectives. The dHPR achieved an $O(1/k)$ iteration complexity for the KKT residual and dual objective error. We applied the sGS technique to decouple the linear operator and consensus constraint, overcoming the limitation of requiring large proximal operators. Extensive numerical experiments confirm the superior performance and broad applicability of dHPR. As potential future research directions, we will investigate the dHPR algorithm for solving the distributed convex and nonconvex composite optimization problems in stochastic and asynchronous settings.

\newpage
\appendix

\section{Proof of Main Theorems}\label{app:thm_convergence}

To prove Theorems~\ref{thm:convergence} and \ref{thm:complexity}, we first establish the following lemma.
\begin{lemma}\label{lem:equiv}
 Suppose that $\mc S_z^i+A_iA_i^\top\succ\bs0$ for all $i\in\mc N$. Given $\bs u^0=(\bs z^0,\bs w^0,\bs v^0,\bs x^0)\in \mathcal{U}$,
the sequence $\{\bs u_s^{k}\}=\{(\bs z^k,\,\bs s^k,\,\bs v^k,\,\bs x^k)\}$ generated by Algorithm~\ref{alg:hprdop} with $ \bs u_s^0=(\bs z^0, \bs U^{\top}\bs w^0,\bs v^0,\bs x^0)$ is equivalent to the sequence $\{(\bs z^k,\,\bs U^\top\bs w^k,\,\bs v^k,\,\bs x^k)\}$, where $\{\bs u^{k}\}=\{(\bs z^k,\,\bs w^k,\,\bs v^k,\,\bs x^k)\}$ is generated by the following updates with the initial point $\bs u^0$:
    \begin{subequations}\label{eq:sgshpr}
    \begin{empheq}[left=\empheqlbrace]{align}
        \label{eq:sgshpr_v}\bar{\bs v}^{k+1}&=\argmin\limits_{\bs v\in\mb R^{Np}} \left\{L_\sigma(\bs z^k,\bs w^k,\bs v;\bs x^k)\right\},\\
            \label{eq:sgshpr_x}\bar{\bs x}^{k+1}&={\bs x}^k-\sigma(\bs A^\top\bs z^k+\bs U^\top\bs w^k+\bar{\bs v}^{k+1}),\\
            \label{eq:sgshpr_w1}\bar{\bs w}^{k+\frac{1}{2}}&=\argmin\limits_{\bs w\in\mb R^{Np}} \left\{L_\sigma(\bs z^k,\bs w,\bar{\bs v}^{k+1};\bar{\bs x}^{k+1})+\dfrac{\sigma}{2}\left\| \bs w-\bs w^k \right\|^2_{\mc S_w}\right\}, \\
            \label{eq:sgshpr_z}\bar{\bs z}^{k+1}&=\argmin\limits_{\bs z\in\mb R^m} \left\{L_\sigma(\bs z,\bar{\bs w}^{k+\frac12},\bar{\bs v}^{k+1};\bar{\bs x}^{k+1})+\dfrac{\sigma}{2}\left\| \bs z-\bs z^k \right\|^2_{\mc S_z}\right\}, \\
            \label{eq:sgshpr_w2}\bar{\bs w}^{k+1}&=\argmin\limits_{\bs w\in\mb R^{Np}} \left\{L_\sigma(\bar{\bs z}^{k+1},\bs w,\bar{\bs v}^{k+1};\bar{\bs x}^{k+1})+\dfrac{\sigma}{2}\left\| \bs w-\bs w^k \right\|^2_{\mc S_w}\right\}, \\
            \hat{\bs u}^{k+1}&=2\bar{\bs u}^{k+1}-\bs u^k,\\
            \bs u^{k+1}&=\dfrac{1}{k+2}\bs u^0+\dfrac{k+1}{k+2}\hat{\bs u}^{k+1},
    \end{empheq}
    \end{subequations}
    where $\mc S_z$ and ${\mc S}_w$ are given in \eqref{eq:Sz&Sw}, and $L_\sigma$ is defined in \eqref{eq:lagrangian}.
\end{lemma}

\begin{proof}
    Clearly, the updates of $\bar{\bs v}^{k+1},\,\bar{\bs x}^{k+1}$ and $\bar{\bs z}^{k+1}$ in Algorithm~\ref{alg:hprdop} are equivalent to \eqref{eq:sgshpr_v}, \eqref{eq:sgshpr_x} and \eqref{eq:sgshpr_z}, respectively. Next, we consider the updates of $\bar{\bs w}^{k+\frac12}$ and $\bar{\bs w}^{k+1}$, i.e., \eqref{eq:sgshpr_w1} and \eqref{eq:sgshpr_w2}.
    Notice that $\mc S_w=\lambda_UI-\bs U^2$, by the optimality condition of \eqref{eq:sgshpr_w1}, the update of $\bar{\bs w}^{k+\frac12}$ is given by
    \begin{align}\label{eq:update_w1}
        \notag\bar{\bs w}^{k+\frac12} &= \frac{\bs U}{\sigma\lambda_U}\left( \bar{\bs x}^{k+1}-\sigma\bs A^\top\bs z^k-\sigma\bar{\bs v}^{k+1}-\sigma\bs U^\top\bs w^k \right)+\bs w^k \\
        &=\bs w^k+\frac1{\sigma\lambda_U}\bs U(2\bar{\bs x}^{k+1}-\bs x^k),
    \end{align}
    where the second equality holds by \eqref{eq:sgshpr_x}. Similarly, substituting $\bs z^k$ with $\bar{\bs z}^{k+1}$ in the above equality, the update of $\bar{\bs w}^{k+1}$ is given by
    \begin{equation}\label{eq:update_w2}
        \bar{\bs w}^{k+1}=\bar{\bs w}^{k+\frac12}+\frac1{\lambda_U}\bs U\bs A^\top(\bs z^k-\bar{\bs z}^{k+1}).
    \end{equation}
    Multiplying $\bs U^\top$ to both sides of \eqref{eq:update_w1} and \eqref{eq:update_w2} and noticing that $\bs s=\bs U^\top\bs w,\,\bs U^2=I-\bs W$, the proof is completed.
\end{proof}

Based on the above lemma, we now give the proof of Theorem~\ref{thm:convergence}.

\begin{proof} Notice that $\mc T=\mc S+\hat{\mc S}$, where $\mc S$ and $\hat{\mc S}$ are given in \eqref{eq:S&Shat}. According to Proposition~\ref{pro:sgs}, for any $k\geq0$, the updates of $\bs z$ and $\bs w$ in \eqref{eq:sgshpr}, i.e., \eqref{eq:sgshpr_w1}-\eqref{eq:sgshpr_w2}, are equivalent to the following update:
\begin{equation}\label{eq:z&w}
\begin{aligned}
    (\bar{\bs z}^{k+1},\,\bar{\bs w}^{k+1})=\argmin\limits_{(\bs z,\,\bs w)\in\mb R^{m}\times\mb R^{Np}} \left\{L_\sigma(\bs z,\bs w,\bar{\bs v}^{k+1};\bar{\bs x}^{k+1})+\frac{1}{2}\left\| (\bs z,\,\bs w)-(\bs z^k,\,\bs w^k) \right\|^2_{\mc T}\right\}.
\end{aligned}
\end{equation}
Moreover, $\mc T+\sigma\bs A_U^\top\bs A_U$ is positive definite. Thus, \eqref{eq:z&w} admits a unique solution. By Lemma~\ref{lem:equiv}, Algorithm~\ref{alg:hprdop} corresponds to Algorithm~\ref{alg:sphpr} with $\mc T=\mc S+\hat{\mc S}$, which corresponds to the accelerated pADMM \cite[Alg. 3.1]{sun2025accpadmm} with $\rho=2$ and $\alpha=2$. Based on \cite[Cor. 3.5]{sun2025accpadmm}, the global convergence result of Algorithm~\ref{alg:hprdop} in Theorem~\ref{thm:convergence} holds. \end{proof}   

Based on Lemma \ref{lem:equiv}, we now give the proof of Theorem~\ref{thm:complexity}.
\begin{proof}
    According to Lemma~\ref{thm:convergence}, Algorithm~\ref{alg:hprdop} corresponds to an HPR method with the semi-proximal term $\mc T=\mc S+\hat{\mc S}$, which is an accelerated pADMM \cite[Alg. 3.1]{sun2025accpadmm} with $\rho=2$ and $\alpha=2$.
    Thus, based on the iteration complexity results for the accelerated pADMM established in \cite[Thm. 3.7]{sun2025accpadmm}, we can derive the iteration complexity of Algorithm~\ref{alg:hprdop} in Theorem~\ref{thm:complexity}.
\end{proof}

\section{Additional Numerical Experiments}\label{app:experiment}

\subsection{dHPR vs dual L-HPR}\label{app:dual-lhpr}

Directly applying Algorithm~\ref{alg:sphpr} to the dual problem \eqref{eq:dualdop} with $\mc T=\sigma(\lambda_{AU}I-\bs A_U^\top\bs A_U)$ leads to the dual linearized-HPR (dual L-HPR) method in Algorithm~\ref{alg:dual-lhpr}.

\begin{algorithm}[ht]
\caption{Dual L-HPR: Dual linearized HPR method for solving the dual problem \eqref{eq:dualdop}}
\label{alg:dual-lhpr}
\begin{algorithmic}[1] 
\Require Let $\mc T=\sigma(\lambda_{AU}I-\bs A_U^\top\bs A_U)$ with $\lambda_{AU}\geq\lambda_{\max}(\bs A_U)$. Choose $\bs u^0=(\bs z^0,\,\bs w^0,\,\bs v^0,\,\bs x^0)\in\mc U$ and set $\sigma>0$.
\For{$k=0,\,1,\,\cdots$}
    \State $\bar{\bs v}^{k+1}=\argmin\limits_{\bs v\in\mb R^{Np}} \left\{L_\sigma(\bs z^k,\bs w^k,\bs v;\bs x^k)\right\}$
    \State $\bar{\bs x}^{k+1}={\bs x}^k-\sigma(\bs A^\top\bs z^k+\bs U^\top\bs w^k+\bar{\bs v}^{k+1})$
    \State
    $(\bar{\bs z}^{k+1},\bar{\bs w}^{k+1})=\argmin\limits_{(\bs z,\bs w)\in\mb R^{m}\times\mb R^{Np}} \left\{L_\sigma(\bs z,\bs w,\bar{\bs v}^{k+1};\bar{\bs x}^{k+1})+\dfrac{1}{2}\left\| (\bs z,\bs w)-(\bs z^k,\bs w^k) \right\|^2_{\mc T}\right\}$
    \State $\hat{\bs u}^{k+1}=2\bar{\bs u}^{k+1}-\bs u^k$
    \State $\bs u^{k+1}=\dfrac{1}{k+2}\bs u^0+\dfrac{k+1}{k+2}\hat{\bs u}^{k+1}$
\EndFor
\State \Return $\bar{\bs u}^{k+1}$
\end{algorithmic}
\end{algorithm}

Fig.~\ref{fig:vslhpr} presents a comparison of dHPR (Algorithm~\ref{alg:hprdop}) and the distributed version of dual L-HPR (Algorithm~\ref{alg:dual-lhpr}) for solving distributed LASSO problems, evaluated on both synthetic ($m=10,\,p=50$) and real-world (UCI $\mathtt{pyrim}$) datasets, where the experimental setup follows that in Section~\ref{sec:numerical}. The curves of relative KKT residual ($\eta_{\text{KKT}}$ defined in \eqref{eq:rekkt}) demonstrate that dHPR consistently outperforms dual L-HPR, exhibiting significantly faster convergence in both cases. This superior performance can be attributed to the implementation of the sGS decomposition technique in dHPR, which yields a smaller proximal operator $\mc S+\hat{\mc S}$ (where $\mc S$ and $\hat{\mc S}$ are defined in \eqref{eq:S&Shat}) compared to dual L-HPR, thereby enhancing computational efficiency.

\begin{figure}[htbp]
\centering
    \begin{subfigure}{0.45\linewidth}
        \includegraphics[width=\linewidth]{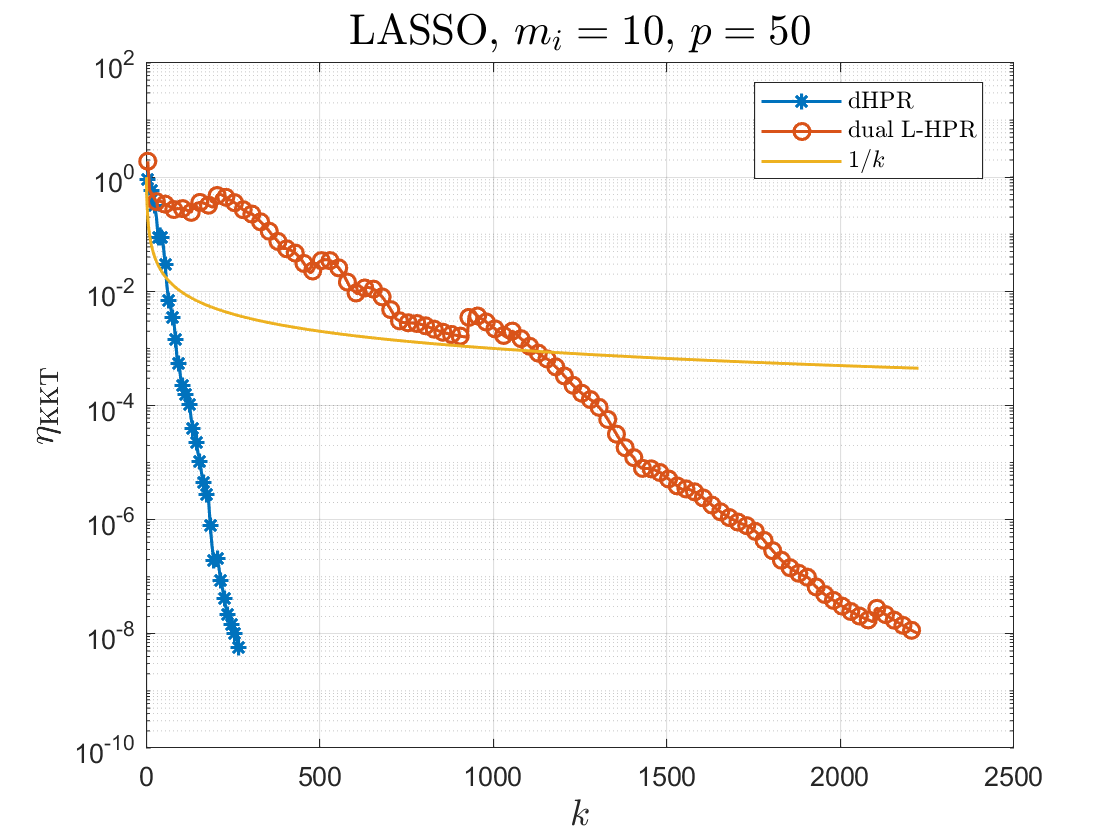}
        \caption{Synthetic data}
        \label{subfig:vslhpr_synthetic}
    \end{subfigure}
    \hfill
    \begin{subfigure}{0.45\linewidth}
        \includegraphics[width=\linewidth]{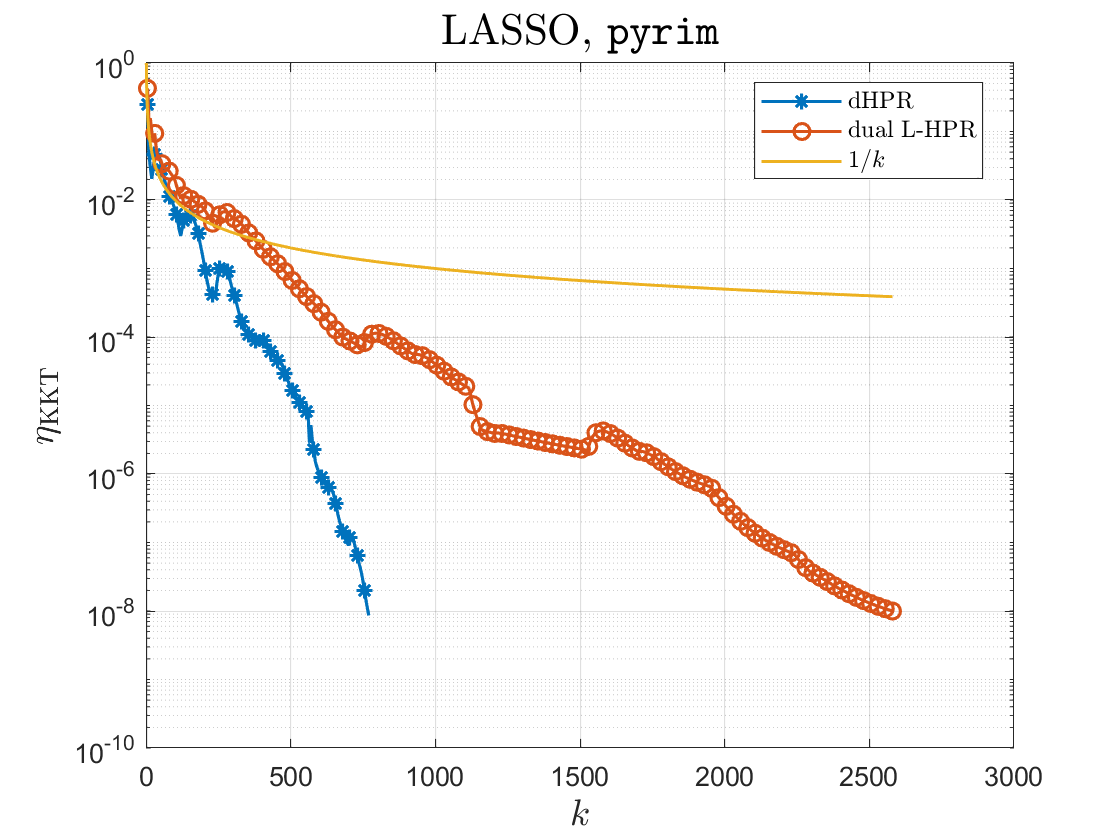}
        \caption{UCI instance}
        \label{subfig:vslhpr_uci}
    \end{subfigure}
\caption{Comparison of dHPR and dual L-HPR}
\label{fig:vslhpr}
\end{figure}

\subsection{Validation of the KKT residual bound in Theorem~\ref{thm:complexity}}\label{app:theorykkt}

According to Theorem~\ref{thm:complexity}, we define the relative KKT residual as
\begin{align}\label{eq:rekkt}
    \eta_{\text{KKT}}:=\max\left\{ 
    \dfrac{\|\prox_f(\bs z+\bs A\bs x)-\bs A\bs x\|}{1+\|\bs z\|+\|\bs A\bs x\|},\,
    \dfrac{\|\bs U\bs x\|}{1+\|\bs x\|},\dfrac{\|\prox_r(\bs v+\bs x)-\bs x\|}{1+\|\bs v\|+\|\bs x\|},\,
    \dfrac{\|\bs A^\top\bs z+\bs s+\bs v\|}{1+\|\bs z\|+\|\bs s\|+\|\bs v\|}
    \right\}.
\end{align}
Fig.~\ref{fig:theorykkt} demonstrates the convergence behavior of dHPR in terms of $\eta_{\text{KKT}}$ for distributed LASSO problems, compared to theoretical baselines on both synthetic ($m=10,\,p=50$) and real-world (UCI $\mathtt{pyrim}$) datasets. The dHPR is implemented following the setup in Section~\ref{sec:numerical} without the restart strategy and the $\sigma$ update. The plots reveal that dHPR exhibits the theoretically predicted $O(1/k)$ iteration complexity in Theorem~\ref{thm:complexity}.

\begin{figure}[htbp]
\centering
    \begin{subfigure}{0.45\linewidth}
        \includegraphics[width=\linewidth]{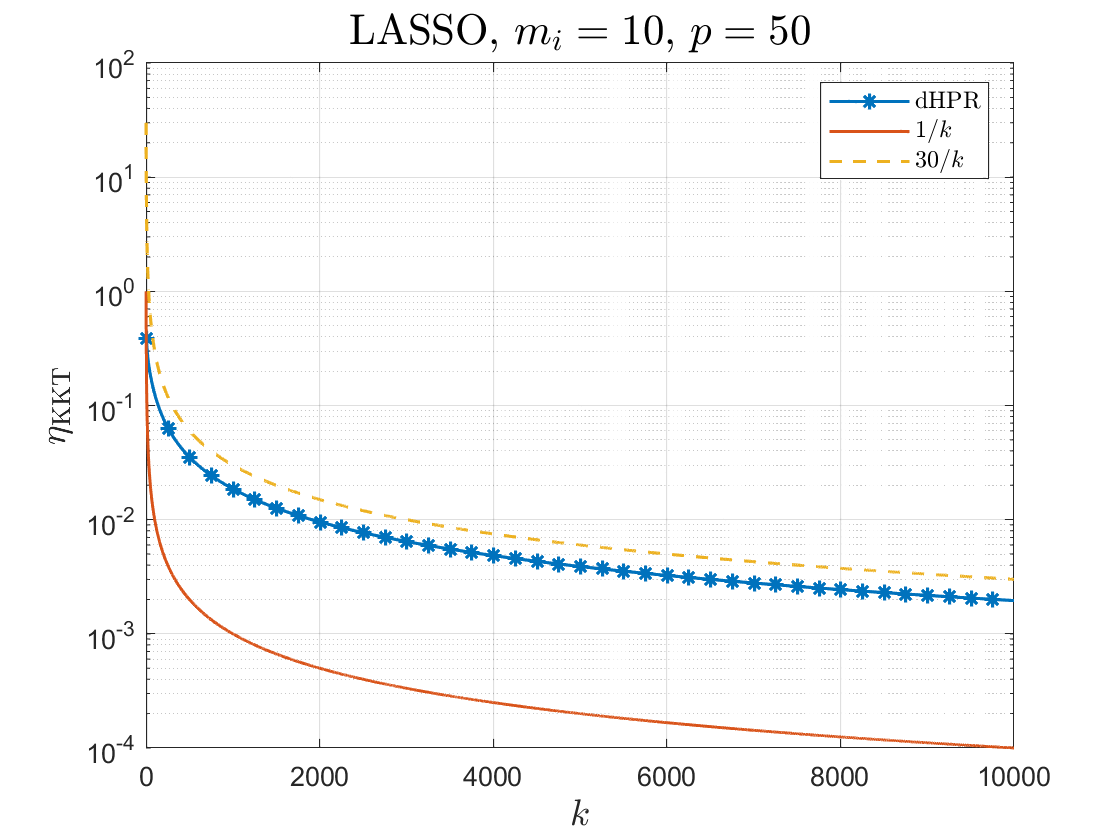}
        \caption{Synthetic data}
        \label{subfig:theorykkt_synthetic}
    \end{subfigure}
    \hfill
    \begin{subfigure}{0.45\linewidth}
        \includegraphics[width=\linewidth]{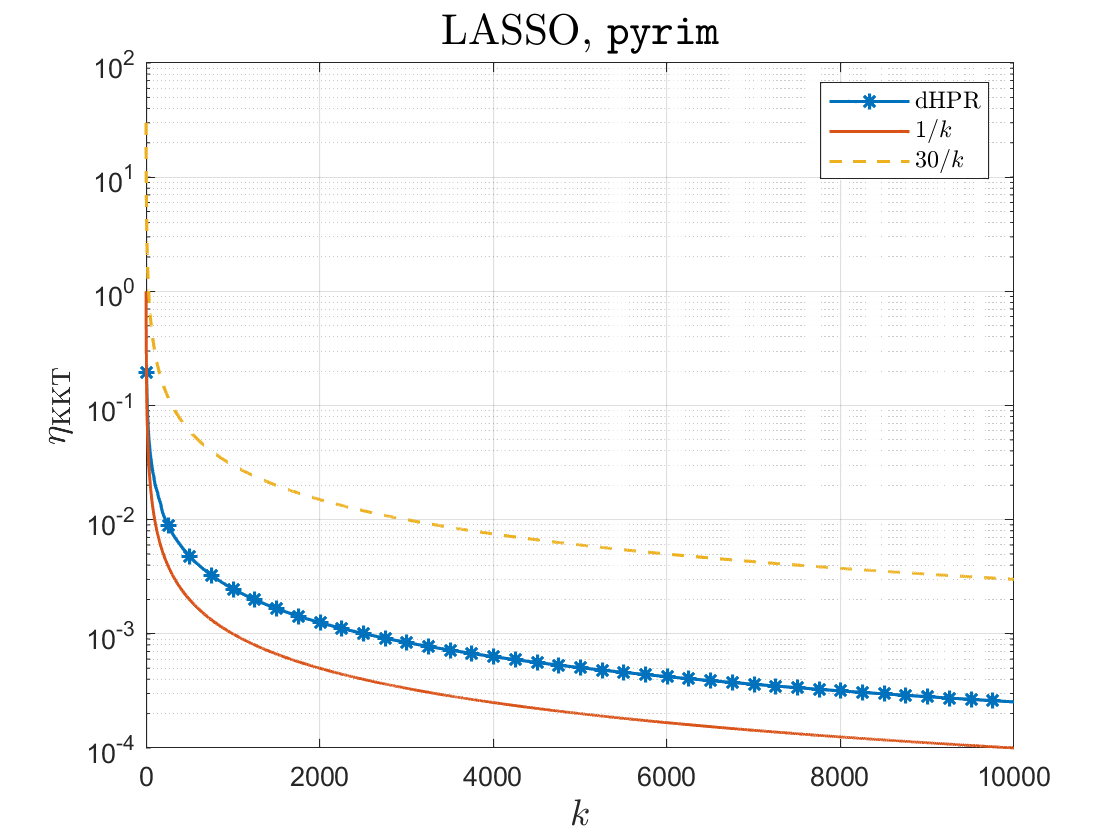}
        \caption{UCI instance}
        \label{subfig:theorykkt_uci}
    \end{subfigure}
\caption{Validation of the theoretical bounds of KKT residual in Theorem~\ref{thm:complexity}}
\label{fig:theorykkt}
\end{figure}

\bibliography{reference}{}
\bibliographystyle{ieeetr}
\end{document}